\documentclass[10pt]{amsart}

\usepackage{color}
\usepackage{amssymb, amscd, amsmath, amsthm,  graphicx, latexsym}
\usepackage{xcolor}
\usepackage[makeroom]{cancel}
\usepackage{multicol}
\usepackage{hyperref}
\usepackage{comment}
\usepackage[T1]{fontenc}
\usepackage{tikz}
\usetikzlibrary{calc,matrix}

\newtheorem{theorem}{Theorem}[section]

\newtheorem{lemma}[theorem]{Lemma}
\newtheorem{corollary}[theorem]{Corollary}
 \newtheorem{manualtheoreminner}{Theorem}
\newenvironment{manualtheorem}[1]{%
  \IfBlankTF{#1}
    {\renewcommand{\themanualtheoreminner}{\unskip}}
    {\renewcommand\themanualtheoreminner{#1}}%
  \manualtheoreminner
}{\endmanualtheoreminner}

\newtheorem{definition}[theorem]{Definition}

\newtheorem{remark}[theorem]{Remark}

\newcommand{\x}{~}
\usepackage{enumitem}
\usepackage{blkarray, relsize}
\makeatletter
\newcommand{\setBAcolsep}[1]{\BA@colsep=#1}
\makeatother

\newcommand*\deltaD[1]{\Delta^D_{#1}}
\newcommand*\deltaC[1]{\Delta^C_{#1}}

\newcommand*\PD[1]{P'(#1)}
\newcommand*\PC[1]{P(#1)}
\newcommand*\pD[1]{p'_{#1}}
\newcommand*\pC[1]{p_{#1}}

\title{There and Back Again with Kaprekar 2-cycles}

\author{Rebekah Mayne}
\author{Quinn Shapiro}
\author{Cornelia A. Van Cott}

\begin{document}

\begin{abstract}
    A curious property of integers, first observed by  D.~R. Kaprekar in 1949, is as follows: write the digits of a 4-digit number in both descending and ascending order, and then find the positive difference between these integers. Iterating this process on any 4 digit number -- except for multiples of 1111 -- eventually produces the number 6174. The process ends here, since 6174 is a fixed point of the procedure. In general, if we start with any integer and iterate this process, it either ends at a fixed point (as with 4 digits) or enters a cycle of numbers, called a Kaprekar cycle. In 2011, Dolan classified all fixed points of this process. We study Kaprekar cycles of length 2. We find general properties of numbers in a Kaprekar 2-cycle, and we classify all 2-cycles such that the two integers have the same smallest digit. 
\end{abstract}
\maketitle

\section{Introduction}
In 1949, D.~R.~Kaprekar, an Indian mathematician who spent more than 40 years studying recreational number theory, observed a curious property of positive integers. Start with a positive integer $N$ with $n$ digits. Let $\overline{N}$ and $\underline{N}$ denote the numbers obtained by writing the digits of $N$ in descending order and ascending order, respectively. Compute $k(N) = \overline{N} - \underline{N}$. By padding $k(N)$ with zeros in front if necessary, $k(N)$ is another $n$ digit number, and so we can repeat the process again with $k(N)$. This is called the Kaprekar process.

Kaprekar observed that when this process is iterated with any four digit number not divisible by 1111, the process eventually yields the four digit number 6174~\cite{kaprekar}. Moreover, observe that $k(6174) = 7641 - 1467 = 6174$. Therefore the process terminates at 6174. For example, beginning with $9542$, we compute:
\begin{align*}
k(9542) & = 9542 - 2459 = 7083\\
k(7083) & = 8730 - 0378 = 8352\\
k(8352) & = 8532 - 2358 = 6174
\end{align*}
In the above example, we reached the value 6174 in three steps. In general, 4-digit numbers reach 6174 in at most seven steps~\cite{devlinzeng}. This phenomenon also occurs when the initial number has 3 digits. In this case, the process terminates at 495 in at most six steps~\cite{eldridge}. 

More generally, the process terminates in one of a collection of so-called {\bf Kaprekar cycles} of length 1 or more~\cite{prichett}. With two-digit numbers, the process ends in the 5-cycle $(90, 81, 63, 72, 54)$. In five-digit numbers, Kaprekar~\cite{kaprekar} found that the process terminates in one of three different cycles: $(53955, 59994)$, $(61974, 82962, 75933, 63954),$ or $(62964, 71973, 83952, 74943).$ With six-digit numbers, the process either ends in the 1-cycle $549945$, the 1-cycle $631764$, or the 7-cycle:
$$
(420876, 851742, 750843,  840852, 860832, 862632, 642654).
$$

In 2011, Stan Dolan classified all Kaprekar cycles of length 1 (that is, positive integers $C$ such that $k(C)=C$, also called {\bf Kaprekar constants})~\cite{dolan}. He found five infinite families of Kaprekar constants. We already mentioned four examples: 495, 6174, 549945, and 631764. The full list is in Theorem~\ref{dolan} below. A comment on notation is needed. The notation, say, $6_x$, denotes $x$ copies of the digit 6. 
\begin{theorem}\cite{dolan}\label{dolan}
    If $C$ is a Kaprekar constant, then $\overline{C}$ is one of the following:
    \begin{enumerate}
        \item $76_{x+1}43_x1$, for $x\geq 0$;
        \item $9_z5_z4_z$, for $z>0$;
        \item $9_z8_z7_z6_{z+y}5_z4_z3_{z+y}2_z1_z$, for $z>0$, and $y\geq 0$;
        \item $9_{z+a}8_{z+x}7_{z+x}6_{z+y}5_{z+x}4_{z+x}3_{z+y}2_{z+x}1_{z+x}0_a$, for $a, x>0$ and $y, z\geq 0$;
        \item $9_z8_{z-2y}7_{z+y}6_{z-y}5_z4_z3_{z-y}2_{z+y}1_{z-2y}$, for $z>2y>0$.
    \end{enumerate}
\end{theorem}

The purpose of our work is to study Kaprekar 2-cycles in base 10. That is, we study pairs of distinct integers $\{C,D\}$ such that $k(C) = D$ and $k(D) = C$. A 2-cycle with 5 digits was already mentioned above $\{53955, 59994\}$. Through a computer search, we discovered infinitely many Kaprekar 2-cycles, which divide into three distinct families. These pairs of integers can be readily checked to be Kaprekar 2-cycles.
\begin{theorem}\label{MainResult}
The following are three families of Kaprekar 2-cycles:
\begin{enumerate}
    \item $\{95553, ~99954\}$;
    \item $\{9_{2}6_{x+1}5_{2}4_{2}3_{x}2_{1}1_{1},  9_{1}8_{2}7_{2}6_{x}3_{x}2_{3}0_{1}\},$ for $x\geq 2$;
    \item 
  $\{9_z8_{z-v-u}7_{z+v}6_{z-v}5_z4_z3_{z-v}2_{z+v}1_{z-v-u},~\\
    \phantom{\{} 9_z8_{z-v-u}7_{z+u}6_{z-u}5_z4_z3_{z-u}2_{z+u}1_{z-v-u}\},$ 
   \\ where $u$, $v$, and $z$ are nonnegative integers, $z>u+v>0$, and $u\neq v.$
   
\end{enumerate}
\end{theorem}

\begin{remark} Notice that if $u=v$ in Family (3) above, then $\overline{C}=\overline{D}$, and we have the Kaprekar constants in Family (5) of Theorem~\ref{dolan}.
\end{remark}
We find general properties of Kaprekar 2-cycles as follows.
\begin{theorem}\label{9}
   Let $\{C, D\}$ be a Kaprekar 2-cycle.
   \begin{enumerate}
   \item Except for the 2-cycle $\{53955, 59994\}$, the digit sum of $C$ equals that of $D$.
   \item Both $C$ and $D$ contain the digit 9.
\end{enumerate}
\end{theorem}

We prove Theorem ~\ref{9}(1) in Section~\ref{4} (see Theorem~\ref{sum}), while Theorem ~\ref{9}(2) is proved in Theorem~\ref{no2cycles}. In addition to these properties, we find the location of the smallest digit in a Kaprekar 2-cycle. This statement requires a bit of notation, so we delay stating it until Theorem~\ref{smallest}. Finally, in Section~\ref{9+samesmallest}, we prove that the Kaprekar 2-cycles in Family (3) from Theorem~\ref{MainResult} are the only 2-cycles for which $C$ and $D$ have the same smallest digit.

\begin{theorem}\label{familyA}
    If $\{C,D\}$ is a Kaprekar 2-cycle where $C$ and $D$ have the same smallest digit, then $\{\overline{C},\overline{D}\}$ is:
    $$ \{9_{z}8_{z-v-u}7_{z+v}6_{z-v}5_z4_z3_{z-v}2_{z+v}1_{z-v-u}, 9_{z}8_{z-v-u}7_{z+u}6_{z-u}5_z4_z3_{z-u}2_{z+u}1_{z-v-u} \} $$
    where $u$, $v$, and $z$ are nonnegative integers such that $z>u+v>0$ and $u\neq v.$
\end{theorem}

The complete classification of Kaprekar 2-cycles $\{C, D\}$ has since been proven by Dolan~\cite{dolan0}. He proved that every Kaprekar 2-cycle belongs to one of the three families in Theorem~\ref{MainResult}.

\subsubsection*{Acknowledgments} We thank Stan Dolan for his helpful correspondence and comments along the way.

\section{Notation for Kaprekar 2-cycles}\label{section2}
We denote the digits of integers in a Kaprekar 2-cycle $\{C,D\}$ as follows:
\[
\overline{C} = y_1~y_2~\cdots~y_n~
\text{ and } 
~\overline{D} = z_1~z_2~\cdots~z_n.
\]
For any integer $\overline{C}$, we denote by $\ell$ the maximal integer such that $y_\ell > y_{n+1-\ell}$. Thus all digits strictly between $y_\ell$ and $y_{n+1-\ell}$ (that is, the middle $n-2\ell$ digits of $\overline{C}$) are equal. 
The values of $\ell$ for the numbers in the Kaprekar 2-cycle $\{53955, 59994\}$ are 1 and 2, respectively. This 2-cycle is special. We prove in Corollary~\ref{ell} that the value of $\ell$ is identical for elements of a Kaprekar 2-cycle with more than 5 digits. Until then, we distinguish these values corresponding to $C$ and $D$ via $\ell_c$ and $\ell_d$, respectively. 

The difference $k(C) = \overline{C} - \underline{C}$ naturally pairs the $i^{th}$ and the $(n+1-i)^{th}$ digit of $\overline{C}$.  For $1\leq i \leq \ell_c$, we call these pairs of digits $y_i$ and $y_{n+1-i}$ {\bf corresponding pairs} of $\overline{C}$. Denote the difference of corresponding pairs for $C$ and $D$ via $\deltaC{i} = y_i - y_{n+1-i}$ and $\deltaD{i} = z_i - z_{n+1-i}$, respectively. Denote the center $n-2\ell_c$ digits of $\overline{C}$ by $x$. If $\{C,D\}$ is a Kaprekar 2-cycle, then the digits of $k(C) = D$ are as follows:
    \begin{equation}\label{D}
    \renewcommand{\arraystretch}{1.3}
    \setlength\arraycolsep{0.8pt}
    \begin{matrix}
        ~ & \overline{C} = & y_1 & y_2 & \cdots & y_{\ell_c} & x \cdots x & y_{n+1-\ell_c} & \cdots & y_{n-1} & y_n  \\
        - & \underline{C}= & y_n & y_{n-1} & \cdots &  y_{n+1-\ell_c} & x \cdots x &  y_{\ell_c} & \cdots  & y_2 & y_1 \\
        \hline
        ~ & D= & \deltaC{1} & \deltaC{2} & \cdots &(\deltaC{\ell_c}-1) & \underbrace{9 \cdots 9}_{n - 2\ell_c} &(9-\deltaC{\ell_c}) & \cdots & (9-\deltaC{2}) & (10 - \deltaC{1})&
    \end{matrix}
    \end{equation}
Similarly, the digits of $C$ are:
\begin{equation}\label{C}
C = \deltaD{1}  \deltaD{2}  \cdots (\deltaD{\ell_d}-1)  \underbrace{9 \cdots 9}_{n - 2\ell_d} (9-\deltaD{\ell_d})  \cdots  (9-\deltaD{2})  (10 - \deltaD{1}). \end{equation}

\section{A special case for Kaprekar 2-cycles}
The Kaprekar 2-cycle $\{53955, 59994\}$ has several unique properties. The following theorem gives one of these defining characteristics.

\begin{theorem}\label{l=1}
Let $\{C,D\}$ be a Kaprekar 2-cycle. If $\ell_c=1$, then $C=53955$ and $D= 59994$.
\end{theorem}
\begin{proof}
     Let $n$ denote the number of digits in $C$ and $D$. Observe that $n\geq 5$, since there are no Kaprekar 2-cycles with 4 or fewer digits. Since $\ell_c=1$, the digits of $\overline{C}$ can be expressed as
     \begin{equation}\label{C-second-eq}\overline{C} = y_1~\underbrace{x\cdots x}_{m}~y_n,\end{equation} 
     where $m=n-2$, $y_1\geq x \geq y_n$, and $y_1>y_n$.
  By Equation~\ref{D}, we have 
 $$D=\overline{C} - \underline{C} =(y_1-y_n-1) \underbrace{9\cdots9}_m  (10+y_n-y_1).$$
So then, $\overline{D} = \underbrace{9\cdots 9}_{m}rs$, where $\{r,s\}=\{ y_1-y_n-1, 10+y_n-y_1\}$ and $r>s$. Note that $r\neq s$, since that would imply  $2(y_1-y_n) = 11$. We can compute $\overline{D} - \underline{D}$ in terms of $r$ and $s$ for all but the first two digits as follows. 
\begin{equation}\label{subtract}
\renewcommand{\arraystretch}{1.2}
\setlength\arraycolsep{2pt}
\begin{matrix}
   ~ & \overline{D}= & 9 & \cancel{9}^{8} &
   \begin{matrix}
     \cancel{9} & \cdots & \cancel{9}
   \end{matrix}  
   & \cancel{r}^{r+9} & \cancel{s}^{s+10} \\
   - & \underline{D} = & s & r \phantom{^{8}} & 
   \begin{matrix}
         9 & \cdots & 9 
   \end{matrix} 
   & 9 \phantom{^{r+9}}& 9 \phantom{^{s+10}}\\
   \hline
   ~ & C= & ? & ? \phantom{^{8}} & 
   \underbrace{\begin{matrix}
       9  & \cdots & 9 
   \end{matrix}}_{m-2}
   & r \phantom{^{r+9}} & (s+1) \phantom{^{s}} 
\end{matrix}
\end{equation}
Since $m-2\geq1$, we conclude that 9 is a digit of $C$. Therefore, $y_1=9$, and the first and last digits of $D$ are $8-y_n$ and $1+y_n$, respectively. Moreover, observe that neither $8-y_n$ nor $1+y_n$ equals 9. For, if $y_n = 8$, then $C$ would either be $99\cdots 98$ or $988\cdots 8$, neither of which yields a Kaprekar 2-cycle. So, $y_n\leq 7$, and $D$ has exactly $m$ digits equal to 9. With this information, we complete the calculation in Equation~\ref{subtract} and find: 
\begin{equation}\label{C4}
C= (9-s) (8-r)  \underbrace{9 \cdots 9}_{m-2}r (s+1).
\end{equation}
We now consider two cases corresponding to the magnitude of $y_n$.

First,  suppose $4 \leq y_n \leq 7$.  In this case, it follows that $1+y_n > 8-y_n$, and so $\overline{D} = 9\cdots 9 (1+y_n)(8-y_n)$. Equation~\ref{C4} with $r = 1+y_n$ and $s = 8-y_n$ implies 
$$C = (1+y_n) (7-y_n) \underbrace{9 \cdots 9}_{m-2} (1+y_n) (9-y_n).
$$
We have a contradiction because none of the digits in this expression equals $y_n$.

Secondly, suppose $y_n \leq 3$. In this case, it follows that $8-y_n > 1+y_n$, and so $\overline{D} = 9\cdots 9 (8-y_n)(1+y_n)$. Equation~\ref{C4} with $r = 8-y_n$ and $s = 1+y_n$ yields 
\begin{equation}\label{C2}
    C = (8-y_n) (y_n) \underbrace{9 \cdots 9}_{m-2} (8-y_n) (2+y_n).
\end{equation}
On one hand, Equation~\ref{C-second-eq} implies that the total number of 9's in $C$ is either $1$ or $m+1$, depending on whether or not $x=9$. On the other hand, Equation~\ref{C2} implies that there must be exactly $m-2$ digits equal to 9. 
Putting the two perspectives together, it follows that $m=3$. Therefore,
$\overline{C} = 9~x~x~x~y_n$ and $C = (8-y_n)~ y_n~ 9~ (8-y_n)~ (2+y_n).$ Because $\overline{C}$ and $C$ have the same digits, it follows that $8-y_n=2+y_n$, and hence $y_n = 3$. Therefore $C = 53955$ and $D = 59994,$ as desired.
\end{proof}

Since $\{53955, 59994\}$ is the only 2-cycle with an integer such that either $\ell_c$ or $\ell_d$ is 1, we henceforth assume $\ell_c, \ell_d>1$, which implies $n>5$. In the following section, we find another property that makes the 2-cycle $\{53955, 59994\}$ unique. 

 \section{The sum of digits for integers in a Kaprekar 2-cycle}\label{4}
The central goal of this section is to prove that, except for the 2-cycle $\{53955, 59994\}$, the two numbers in Kaprekar 2-cycles have the same digit sum. In service of this goal, we begin by considering the location of the smallest digit in a Kaprekar 2-cycle. Our argument has the same structure as Lemma 3 in~\cite{dolan}.  

\begin{theorem}\label{smallest}
    Let $s$ be the smallest digit in the Kaprekar 2-cycle $\{C,D\}$. Then
    \begin{enumerate}
        \item $s$ appears the same number of times in $C$ and $D$, or
        \item $s$ appears exactly once in either $C$ or $D$ and not in the other.
    \end{enumerate}
    Moreover, if $s$ appears at least once as a digit in $D$, then $s$ must be the $\ell_c^{th}$ digit of $D$. The corresponding statement holds for $C$.
\end{theorem}

\begin{proof}
    Without loss of generality, suppose that the digit $s$ occurs at least once in $D$. The number $D$ has the form given by Equation~\ref{regionA}. We aim to show that $s$ is in one or both of the two locations indicated below:
    \begin{equation}\label{regionA}
    D =\deltaC{1}  \cdots  \deltaC{\ell_c-1}  \underbrace{(\deltaC{\ell_c}-1)}_{\text{$\ell_c^{th}$ digit}}  9 \cdots 9 \underbrace{(9-\deltaC{\ell_c})  (9 - \deltaC{\ell_c-1})  \cdots  (9-\deltaC{2})}_{\text{Region A}}  (10 - \deltaC{1})
\end{equation}
    Notice that $10-\deltaC{1} = 10 - (y_1-y_n) \geq 10 - (9-s) = 1 + s>s$. So, the last digit of $D$ cannot equal $s$. Also notice that $\deltaC{1} \geq \deltaC{2} \geq \cdots \geq \deltaC{\ell_c-1} > \deltaC{\ell_c} -1$, thus the digit $s$ must be the $\ell_c^{th}$ digit or in Region A, as indicated above. 
    
    Suppose $s$ is in Region A of $D$. We will show that in this case, $s$ appears the same number of times in $C$ and $D$. Let $i$ be the greatest integer where $2\leq i \leq \ell_c$ such that $9-\deltaC{i}=s$. So in that case, there are exactly $i-1$ digits in Region A of $D$ which equal $s$. Then: 
    \begin{align*}
     9-(y_i-y_{n+1-i}) &= s\\
    y_{n+1-i} - s &= y_i - 9
    \end{align*} 
    On one hand, we know $y_{n+1-i} - s \geq 0$. At the same time, $y_i - 9 \leq 0$. Therefore, $y_{n+1-i} = s$ and $y_i=9$.    
    Since $y_{j} \geq y_i$ for $j<i$, we know $y_{j} = 9$ for all $1\leq j \leq i$. Similarly, $y_{n+1-j} = s$ for all $1\leq j \leq i$. So,
    at least $i$ digits in $C$ equal $s$ (namely, the digits $y_{n+1-i},y_{n+2-i},\ldots,y_n$) and at least $i$ digits equal 9 (the digits $y_1, \ldots, y_i$).

   As with $D$, there are only 2 possible locations for $s$ in $C$: (1)  the $\ell_d^{th}$ digit or (2) Region A of $C$. The $\ell_d^{th}$ digit of $C$ can account for at most one of these digits. So $C$ has at least $i-1$ digits equal to $s$ in Region A. Note that $i-1\neq 0$ since $i\geq2$.
    
    Now we turn the argument around (switching the roles of $C$ and $D$) and conclude that $s$ occurs as a digit in $D$ at least $i$ times. However, we already know that $s$ occurs exactly $i-1$ times in Region A in $D$, so it must be that the $\ell_c^{th}$ digit of $D$ is $s$, and $D$ has exactly $i$ digits equal to $s$.  We said that $C$ had at least $i$ digits equal to $s$. If $C$ had more than $i$ digits equal to $s$, there would be more than $i-1$ digits in Region A of $C$, but that would force $D$ to have more than $i-1$ digits in its Region A. Thus, $C$ has exactly $i$ digits equal to $s$. We have shown that $s$ appears the same number of times in $C$ and $D$.
    
    Now suppose $s$ does {\em not} occur in Region A of $D$. Then, $s$ occurs exactly once in $D$ as the $\ell_c^{th}$ digit. Observe that $s$ cannot appear in Region A of $C$ since that would imply that $s$ also occurs in Region A of $D$. Therefore, there are two possibilities. Either $s$ occurs only as  the $\ell_d^{th}$ digit of $C$ or $s$ does not occur in $C$ at all. The former option results in case (1) of the theorem, while the latter option results in case (2) of the theorem.  In either case, observe from our discussion that whenever the digit $s$ occurs in $D$, it must be the $\ell_c^{th}$ digit, which proves the final statement of the theorem.
\end{proof}

\begin{lemma}\label{morethan2}
If $C$ is in a Kaprekar 2-cycle, then $C$ has three or more distinct digits.
\end{lemma}
\begin{proof}
         Suppose $C$ has exactly two distinct digits. Equation~\ref{C} indicates that $C$ includes the digits: $\deltaD{1}$ and $\deltaD{\ell_d}-1$. Notice that $\deltaD{1}>\deltaD{\ell_d}-1$ because $\deltaD{1} \geq \deltaD{2} \geq \cdots \geq \deltaD{\ell_d-1}\geq \deltaD{\ell_d} > \deltaD{\ell_d} -1$. Thus, all digits of $C$ either equal $\deltaD{1}$ or $\deltaD{\ell_d}-1$. In addition, $\deltaD{1} = \deltaD{i}$ for all $1<i<\ell_d$ because each is a digit in $C$.

    Now consider the digit $9-\deltaD{\ell_d}$ in $C$ (Equation~\ref{C}). Suppose $9-\deltaD{\ell_d}=\deltaD{\ell_d}-1$. Then $\deltaD{\ell_d}=5$. Another digit in $C$ is $10-\deltaD{1}$. Thus, $10-\deltaD{1}= 4$ or $\deltaD{1}$, which implies that $\deltaD{1}=6$ or $5$, respectively. In light of Equation~\ref{C}, $C=\underbrace{6\cdots 64}_{\ell_d}\underbrace{44\cdots 4}_{\ell_d}$ or $\underbrace{5\cdots 54}_{\ell_d}\underbrace{44\cdots 45}_{\ell_d}$, respectively. In either case, $\overline{C}-\underline{C}$ gives a number $D$ such that $\overline{D}-\underline{D} \neq C$, a contradiction. Therefore, $9-\deltaD{\ell_d}\neq\deltaD{\ell_d}-1$.

    The only remaining possibility is $9-\deltaD{\ell_d}=\deltaD{1}$. In this case, the digit $10-\deltaD{1}$ cannot equal $\deltaD{\ell_d}-1$, as it leads to an immediate contradiction. Thus $10-\deltaD{1}=\deltaD{1}$, and we conclude $\deltaD{1}=5$ and $\deltaD{\ell_d}=4$. In that case, $C=\underbrace{5\cdots 53}_{\ell_d}\underbrace{54\cdots 45}_{\ell_d}$. The only way for this number to have two digits is if $\ell_d=2$ and $C = 5355$. But, $C$ does not produce a 2-cycle. Therefore, $C$ must have three or more distinct digits.
\end{proof}

\begin{lemma}\label{notallsame}
    Let $\{C,D\}$ be a Kaprekar 2-cycle. Then, $\deltaC{1}\neq \deltaC{\ell_c}$. 
\end{lemma}

\begin{proof}
Suppose that $\deltaC{1}=\deltaC{\ell_c}$. Then $\deltaC{1}=\deltaC{2}=\cdots = \deltaC{\ell_c}$, and $\overline{C}$ is of the form 
    $\overline{C} = \underbrace{u\cdots u}_{\ell_c}\underbrace{v\cdots v}_{n-2\ell_c} \underbrace{w\cdots w}_{\ell_c},$
    where $u> v> w$ and $n-2\ell_c>0$ by Lemma~\ref{morethan2}. Observe that this implies $D$ contains the digit 9. 
   We consider two cases according to whether $C$ contains the digit 9. 
   
   First, suppose $C$ contains the digit 9. Then $u=9$ and 
    $
    \overline{C} = \underbrace{9\cdots 9}_{\ell_c}\underbrace{v\cdots v}_{n-2\ell_c} \underbrace{w\cdots w}_{\ell_c}.
    $
    Subtracting $\overline{C}-\underline{C}$, we find
    $
    D = \underbrace{(9-w) \cdots (9-w)}_{\ell_c-1}\;  (8-w) \; \underbrace{9 \cdots 9}_{n-2\ell_c}\; \underbrace{w \cdots w}_{\ell_c-1} \; (w+1).
    $
By Theorem~\ref{smallest}, there are two possibilities for the smallest digit in the Kaprekar 2-cycle $\{C,D\}$. The first possibility is that the smallest digit is $w$, which forces $w=8-w$. Hence $w=4$. The second possibility is that the smallest digit is $8-w$ and $8-w<w$, in which case $w>4$. 
    
    In the first case where $w=4$,
    $
    D = \underbrace{5 \cdots 5}_{\ell_c-1}\;  4 \; \underbrace{9 \cdots 9}_{n-2\ell_c}\; \underbrace{4 \cdots 4}_{\ell_c-1} \; 5.
    $
    The first digit of  $\overline{D}-\underline{D}=C$ is 5, so then $v = 5$, and
  $
    \overline{C} = \underbrace{9 \cdots 9}_{\ell_c} \; \underbrace{5\cdots 5}_{n-2\ell_c}\; \underbrace{4 \cdots 4}_{\ell_c}.
    $
    We claim $n-2\ell_c = \ell_c$. For, if $n-2\ell_c > \ell_c$ (respectively, $n-2\ell_c < \ell_c$), then $\overline{D} - \underline{D} = C$ would contain the digit 3 (respectively, 0). 
    Therefore $n-2\ell_c = \ell_c$ and $C=D$, which means we have found a Kaprekar constant (Family 2 in Theorem~\ref{dolan}), rather than a Kaprekar 2-cycle. Thus $w=4$ is not possible.

    The second possibility is that the smallest digit is $8-w$ and  $8-w<w$, which implies $w>4$. Since $D$ contains the digit 9, $\deltaD{1} = z_1 - z_n = 9-(8-w) = w+1$, and hence the last digit of $C$ is $10-\deltaD{1} = 9-w$. Because $w>4$, the digit $9-w$ is less than 5 and is smaller than $w$, but $w$ was the smallest digit in $C$, a contradiction. 
    
    Thus we have shown that $C$ cannot contain the digit 9.
Recall that
$
    \overline{C} = \underbrace{u\cdots u}_{\ell_c}\underbrace{v\cdots v}_{n-2\ell_c} \underbrace{w\cdots w}_{\ell_c},
    $
    where now $9>u> v> w$. Letting $a=u-w$, we have 
    \[
    D=  \underbrace{a \cdots a}_{\ell_c-1}\;  (a-1) \; \underbrace{9 \cdots 9}_{n-2\ell_c}\; \underbrace{(9-a)  \cdots  (9-a)}_{\ell_c-1} \; (10-a) 
    \]
If the smallest digit in the Kaprekar cycle is in both $C$ and $D$, then $w=a-1$ by Theorem~\ref{smallest}. Because $w$ occurs more than once (for, otherwise, $\ell_c=1$), this implies $a-1=9-a$, so $a=5$. This forces $w=4$ and $u=9$, a contradiction. 

If the smallest digit occurs only once in one of $C$ or $D$, it must be in $D$, (again, otherwise $\ell_c=1$). By Theorem~\ref{smallest}, $a-1$ is the smallest digit, $a-1<w$, and $a-1<9-a$, which implies $a<5$.

Observe that $ \overline{D}$ begins with 9 and ends with $a-1$, which means $C =  \overline{D}-\underline{D}$ begins with $10-a$ and ends with $a$. 
Because $a$ is a digit in $C$ and because $a-1$ is the smallest digit in the cycle and is only found in $D$, it follows that $a=w$. Furthermore, since $a=u-w$, it follows that $u=2w$, and the integer $C$ has only three distinct digits: $u = 2w$, $v$, $w$. Now $C$ begins with the digit $10-a=10- w$, so then $10-w$ is one of the three digits of $C$: $2w$, $v$, or $w$. 

We consider each possibility. If $10-w=w$, then $w=5$, forcing $u=10$, a contradiction. If $10-w=2w$, then $3w=10$, but $w$ is an integer. This leaves $10-w=v$ as the only option. Since $u=2w> v$, this implies $2w> 10-w$, so $w>\frac{10}{3}$. Since $2w$ also has to be a single digit, $w=4$, $u=8$, and $v=6$. Therefore: 
\[ 
\overline{C}= \; \underbrace{8 \cdots 8}_{\ell_c} \; \underbrace{6 \cdots 6}_{n-2\ell_c} \; \underbrace{4 \cdots 4}_{\ell_c} \;\; 
\text{and} \;\; D = \underbrace{4 \cdots 4}_{\ell_c-1} \; 3 \; \underbrace{9 \cdots 9}_{n-2\ell_c} \; \underbrace{5 \cdots 5}_{\ell_c-1} \; 6
\]
The second digit of $\overline{D}$ is either 9 or 6, and the second-to-last digit of $\overline{D}$ is 4. Their difference $\deltaD{2} = z_2-z_{n-1}$ is then either 5 or 2. Thus $C$ ought to have a digit equal to 5, 2, or 1. None of these digits are in $C$, a contradiction. Therefore, it must be the case that $\deltaC{1}\neq \deltaC{\ell_c}$. 
\end{proof}

In Section~\ref{section2}, we discussed that pairing $i^{th}$ and $(n+1-i)^{th}$ digit of $\overline{C}$ for $1\leq i\leq \ell_c$ gives the so-called corresponding pairs of $\overline{C}$. We now introduce a second and third way to pair digits.

The second digit pairing comes from again considering the number $\overline{C}$. Observe that the first $n-2\ell_d$ digits of $\overline{C}$ are 9's. Disregard those first $n-2\ell_d$ digits. Pair the remaining digits of $\overline{C}$, starting with the outside pair and working inward. The $(n-2\ell_d+i)^{th}$ and $(n+1-i)^{th}$ digits of $\overline{C}$ form a {\bf rainbow pair} in $\overline{C}$, where $1\leq i \leq \ell_d$. Note that if $\overline{C}$ does not contain the digit 9, then the rainbow and corresponding pairs coincide.

The third pairing arises from considering the integer $C$, rather than $\overline{C}$. Let $r_i$ denote the $i^{th}$ digit of $C$ for $1\leq i\leq \ell_d$. We call the pairs of digits $r_i$ and $r_{n+1-i}$ a {\bf symmetric pair} of $C$. Considering Equation~\ref{C}, exactly one symmetric pair sums to 10, one sums to 8, and the rest sum to 9. 
\begin{equation}\label{symmetricpairs}
r_i + r_{n+1-i}=
\begin{cases}
10, & \text{ if } i = 1\\
9, & \text{ if } 1 < i < \ell_d\\
8, & \text{ if } i = \ell_d
\end{cases}
\end{equation}

We illustrate these three types of pairs with an example. Consider the Kaprekar 2-cycle with $n=13$ digits: $C=9665429654331$, $D=8733209876622$. Observe that $\ell_d =6$, so there are 6 pairs in each case. 
\begin{itemize}
\item The corresponding pairs of $\overline{C}$: $\{9,1\}$, $\{9,2\}$, $\{6,3\}$, $\{6,3\}$, $\{6,4\}$, $\{5,4\}$.
\item The symmetric pairs of $C$: $\{9,1\}$, $\{6,3\}$, $\{6,3\}$, $\{5,4\}$, $\{5,4\},$ $\{6,2\}$.
\item The rainbow pairs of $\overline{C}$ are the same as the symmetric pairs.
\end{itemize}
These three sets of pairs are illustrated below. From left to right, we see the corresponding pairs of $\overline{C}$, the symmetric pairs of $C$, and the rainbow pairs of $\overline{C}$. Unpaired digits are in bold.
$$
     \begin{tikzpicture}[>=stealth,baseline,anchor=base,inner sep=0pt]
      \matrix (foil) [matrix of math nodes,nodes={minimum height=0.5em}] {
        \overline{C} =~ & 9 & 9 & 6& 6 & 6 &5 & \bold{5} &4 &4 &3 &3 &2 &1 \\
      };
      \path ($(foil-1-2.north)+(0,0.5ex)$)   edge[red,bend left=50]    ($(foil-1-14.north)+(0,0.5ex)$)
                ($(foil-1-3.north)+(0,0.5ex)$)   edge[blue,bend left=50]  ($(foil-1-13.north)+(0,0.5ex)$)
                ($(foil-1-4.north)+(0,0.5ex)$) edge[red,bend left=50]      ($(foil-1-12.north)+(0,0.5ex)$)
                ($(foil-1-5.north)+(0,0.5ex)$) edge[blue,bend left=45] ($(foil-1-11.north)+(0,0.5ex)$)
                ($(foil-1-6.north)+(0,0.5ex)$) edge[red,bend left=45] ($(foil-1-10.north)+(0,0.5ex)$)
                ($(foil-1-7.north)+(0,0.5ex)$) edge[blue,bend left=45] ($(foil-1-9.north)+(0,0.5ex)$);
    \end{tikzpicture}
\hspace{1cm}
   \begin{tikzpicture}[>=stealth,baseline,anchor=base,inner sep=0pt]
      \matrix (foil) [matrix of math nodes,nodes={minimum height=0.5em}] {
        C =~ & 9 & 6 & 6 & 5  & 4 & 2 & \bold{9} & 6 & 5 & 4 & 3 & 3 & 1 \\
      };
      \path ($(foil-1-2.north)+(0,0.5ex)$)   edge[red,bend left=50]    ($(foil-1-14.north)+(0,0.5ex)$)
                ($(foil-1-3.north)+(0,0.5ex)$)   edge[blue,bend left=50]  ($(foil-1-13.north)+(0,0.5ex)$)
                ($(foil-1-4.north)+(0,0.5ex)$) edge[red,bend left=50]      ($(foil-1-12.north)+(0,0.5ex)$)
                ($(foil-1-5.north)+(0,0.5ex)$) edge[blue,bend left=45] ($(foil-1-11.north)+(0,0.5ex)$)
                ($(foil-1-6.north)+(0,0.5ex)$) edge[red,bend left=45] ($(foil-1-10.north)+(0,0.5ex)$)
                ($(foil-1-7.north)+(0,0.5ex)$) edge[blue,bend left=45] ($(foil-1-9.north)+(0,0.5ex)$);
    \end{tikzpicture}
\hspace{1cm}
    \begin{tikzpicture}[>=stealth,baseline,anchor=base,inner sep=0pt]
      \matrix (foil) [matrix of math nodes,nodes={minimum height=0.5em}] {
        \overline{C} =~ & \bold{9} & 9 & 6& 6 & 6 &5 & 5 &4 &4 &3 &3 &2 &1 \\
      };
      \path ($(foil-1-3.north)+(0,0.5ex)$)   edge[red,bend left=50]    ($(foil-1-14.north)+(0,0.5ex)$)
                ($(foil-1-4.north)+(0,0.5ex)$)   edge[blue,bend left=50]  ($(foil-1-13.north)+(0,0.5ex)$)
                ($(foil-1-5.north)+(0,0.5ex)$) edge[red,bend left=50]      ($(foil-1-12.north)+(0,0.5ex)$)
                ($(foil-1-6.north)+(0,0.5ex)$) edge[blue,bend left=45] ($(foil-1-11.north)+(0,0.5ex)$)
                ($(foil-1-7.north)+(0,0.5ex)$) edge[red,bend left=45] ($(foil-1-10.north)+(0,0.5ex)$)
                ($(foil-1-8.north)+(0,0.5ex)$) edge[blue,bend left=45] ($(foil-1-9.north)+(0,0.5ex)$);
    \end{tikzpicture}
$$
In this particular example, the symmetric and rainbow pairs coincide. We  pin down exactly when this happens in the following lemma. 
\begin{lemma}\label{inpairs}
Let $\{C,D\}$ be a Kaprekar 2-cycle. Let $\{p, q\}$ and $\{r,s\}$ be the two symmetric pairs of $C$ which sum to 10 and 8, respectively. Without loss of generality, let $p\geq q$ and $r\geq s$.  
\begin{enumerate}
\item The symmetric pairs of $C$ and rainbow pairs of $\overline{C}$ coincide if and only if $p\geq r\geq s\geq q~$ or $~r\geq p\geq q\geq s$.
\item The following are equivalent: (a) the symmetric pairs of $C$ and rainbow pairs of $\overline{C}$ do not coincide, (b) $p>r>q>s$, and (c) $\deltaD{1} +\deltaD{\ell_d} = 10$. 

In this case, all rainbow pairs sum to 9. Moreover, the symmetric pairs $\{p,q\}$ and $\{r,s\}$ of $C$ become rainbow pairs of $\overline{C}$ as follows: $\{p,s\}$ and $\{q,r\}$. All other symmetric pairs are identical to the rainbow pairs. 

\end{enumerate}
\end{lemma}
\begin{remark}
The statements above are given in terms of $C$ but also hold for $D$.
\end{remark}

\begin{proof}
Suppose the values in the two pairs $\{p, q\}$ and $\{r,s\}$ are nested (meaning, $p\geq r\geq s\geq q~$ or $~r\geq p\geq q\geq s$). Let us consider how the digits in the symmetric pairs of $C$ reposition in the formation of $\overline{C}$. Recall that the symmetric pairs of $C$ other than $\{p,q\}$ and $\{r,s\}$ sum to 9 (Equation~\ref{symmetricpairs}). Suppose that $C$ has $t_0$ symmetric pairs of the form $\{0, 9\}$, $t_1$ pairs of the form $\{1, 8\}$, $t_2$ pairs of the form $\{2, 7\}$, $t_3$ pairs of the form $\{3,6\}$, and $t_4$ pairs of the form $\{4,5\}$. To construct $\overline{C}$, we begin by ordering the digits of these symmetric pairs that sum to nine from largest to smallest and append the $n-2\ell_d$ 9's to the beginning. We have:
    \[
   \underbrace{9\cdots9}_{n-2\ell_d} \underbrace{9\cdots9}_{t_0} \underbrace{8\cdots8}_{t_1} \underbrace{7\cdots7}_{t_2} \underbrace{6\cdots6}_{t_3} \underbrace{5\cdots5}_{t_4} \underbrace{4\cdots4}_{t_4} \underbrace{3\cdots3}_{t_3} \underbrace{2\cdots2}_{t_2} \underbrace{1\cdots1}_{t_1} \underbrace{0\cdots0}_{t_0}.
    \]
    To complete the construction of $\overline{C}$, it remains to insert $p, q, r$, and $s$ into the above string. 
    
Since the pairs $\{p, q\}$ and $\{r,s\}$ sum to 10 and 8, respectively, and their magnitudes are nested, the values $p, q, r$, and $s$ fit into the above string and remain rainbow pairs in $\overline{C}$. For example, if $\{p,q\}=\{\textcolor{red}{8,2}\}$ and $\{r,s\}=\{\textcolor{blue}{5,3}\}$, then
    \[
    \overline{C} = \underbrace{9\cdots9}_{n-2\ell_d} \underbrace{9\cdots9}_{t_0} \underbrace{8\cdots8}_{t_1}\textcolor{red}{\bf 8} \underbrace{7\cdots7}_{t_2} \underbrace{6\cdots6}_{t_3}{\textcolor{blue}{\bf 5}} \underbrace{5\cdots5}_{t_4} \underbrace{4\cdots4}_{t_4}{\textcolor{blue}{\bf 3}} \underbrace{3\cdots3}_{t_3} \underbrace{2\cdots2}_{t_2}\textcolor{red}{\bf 2} \underbrace{1\cdots1}_{t_1} \underbrace{0\cdots0}_{t_0}.
    \]
 Observe that all the symmetric pairs from $C$ remain paired as rainbow pairs in $\overline{C}$ above, including the pairs $\{p, q\}$ and $\{r,s\}$. Therefore, the symmetric pairs of $C$ coincide with the rainbow pairs of $\overline{C}$. 
 
Conversely, suppose that the digits $\{p, q\}$ and $\{r,s\}$ are {\em not} nested (meaning, neither of the following hold: $p\geq r\geq s\geq q~$ or $~r\geq p\geq q\geq s$). Observe that it is impossible for $p\geq q > r\geq s$. For, if this were the case, it would follow that $p=q=5$ and $r=s=4$ because $p+q=10$ and $r+s=8$. This, in turn, would imply $\deltaD{1} = \deltaD{\ell_d} = 5$ (since $\{p,q\} = \{\deltaD{1}, 10-\deltaD{1}\}$ and $\{r,s\} = \{\deltaD{\ell_d}-1, 9-\deltaD{\ell_d}\}$), which contradicts Lemma~\ref{notallsame}. Furthermore, note that since $p+q>r+s$, it is also impossible for $r>p>s>q$. Therefore, the only remaining possibility is that the digits of the symmetric pairs $\{p, q\}$ and $\{r,s\}$ are interleaved with each other as follows: $p>r>q>s$. We will prove that if $p>r>q>s$, then the symmetric pairs of $C$ and the rainbow pairs of $\overline{C}$ do not coincide. This will conclude the proof of (1) via a contrapositive argument and will also prove (b) implies (a) in part (2).

Suppose $p>r>q>s$. Because $p+q=10$ and $r+s=8$, it follows that
     $$2 = (p+q) - (r+s) = (p-r) +(q-s).$$ 
     Therefore, $p - r=1$ and $q -s=1$, which implies that $p+s=9$ and $q+r = 9$.    When inserted into the string of digits for $\overline{C}$, the fact that $p=r+1$ and $q=s+1$ and the fact that the pairs $\{p,s\}$ and  $\{q,r\}$ sum to 9 causes $\{p,s\}$ and $\{q,r\}$ to form rainbow pairs, while the remaining rainbow pairs of $\overline{C}$ coincide with the symmetric pairs in $C$. To illustrate, consider that if $\{p,q\}=\{\textcolor{red}{8,2}\}$ and $\{r,s\}=\{\textcolor{blue}{7,1}\}$, then:
    \[
    \overline{C} = \underbrace{9\cdots9}_{n-2\ell_d} \underbrace{9\cdots9}_{t_0} \underbrace{8\cdots8}_{t_1}\textcolor{red}{\bf 8}{\textcolor{blue}{\bf 7}} \underbrace{7\cdots7}_{t_2} \underbrace{6\cdots6}_{t_3} \underbrace{5\cdots5}_{t_4} \underbrace{4\cdots4}_{t_4} \underbrace{3\cdots3}_{t_3} \underbrace{2\cdots2}_{t_2} \textcolor{red}{\bf 2}{\textcolor{blue}{\bf 1}} \underbrace{1\cdots1}_{t_1} \underbrace{0\cdots0}_{t_0}.
    \]
    Note that the symmetric pairs of $C$ coincide with the rainbow pairs of $\overline{C}$ except for the symmetric pairs $\{8,2\}$ and $\{7,1\}$, which switch to rainbow pairs $\{8,1\}$ and $\{7,2\}$ above.  Thus we have shown that the symmetric pairs of $C$ and the rainbow pairs of $\overline{C}$ do not coincide. This concludes the proof of (1) and also proves (b) implies (a) in part (2). 
  
     Next we show (a) implies (c). Suppose that the symmetric pairs of $C$ and the rainbow pairs of $\overline{C}$ do not coincide. By our discussion above, $p+s=9$, and $q+r = 9$. Subtracting the two equations, we find $p-q = r-s$.
     Recall that $\{p,q\} = \{\deltaD{1}, 10-\deltaD{1}\}$ and $\{r,s\} = \{\deltaD{\ell_d}-1, 9-\deltaD{\ell_d}\}$. Thus, the equation $p-q = r-s$ can be restated in one of two ways. Either $10-2\deltaD{1} = -(10-2\deltaD{\ell_d})$ or else  $10-2\deltaD{1} = 10-2\deltaD{\ell_d}$. The latter option simplifies to $\deltaD{1} = \deltaD{\ell_d}$, which contradicts Lemma \ref{notallsame}. The former option simplifies to $\deltaD{1} + \deltaD{\ell_d} = 10$, as desired. Thus (a) implies (c). 
     
     We now show that (c) implies (b). Suppose that $\deltaD{1} + \deltaD{\ell_d} = 10$. Using this, the digits in the first and $\ell_d^{th}$ symmetric pairs of $C$ are: $\{p,q\}=\{\deltaD{1}, 10-\deltaD{1}\}$ and $\{r,s\}=\{9-\deltaD{1}, \deltaD{1}-1\}$. Observe that the digits in these pairs have the following relative magnitudes:
$$\deltaD{1} > \deltaD{1}-1 > 10-\deltaD{1}> 9-\deltaD{1},$$ 
where the middle inequality holds because $\deltaD{1}> \deltaD{\ell_d}$ and $\deltaD{1} + \deltaD{\ell_d} = 10$ imply that $\deltaD{1}>5$. Thus, $p>r>q>s$, and we have proved that (c) implies (b). 
     \end{proof}

\begin{theorem}\label{sum}
    Except for $\{53955, 59994\}$, the two numbers in a Kaprekar 2-cycle have the same digit sum.
\end{theorem}

\begin{proof}
Suppose $\{C,D\}$ is a Kaprekar 2-cycle other than $\{53955, 59994\}$. Theorem~\ref{l=1} implies $\ell_d, \ell_c \geq 2$. By Equations~\ref{D} and~\ref{C}, the digit sums of $C$ and $D$ are $9(n-\ell_d)$ and $9(n-\ell_c)$, respectively. Suppose that $C$ and $D$ do not have the same digit sum. Without loss of generality, this implies $\ell_c > \ell_d$.

In light of Equations~\ref{D} and~\ref{C}, there is a center block of 9's in $C$ and $D$ with $n-2\ell_d$ and $n-2\ell_c$ digits, respectively. Let $u=n-2\ell_c$ denote the size of the block in $D$. The number of digits outside these center blocks of 9's in $C$ and $D$ is even ($2\ell_d$ and $2\ell_c$, respectively), so the size of the center block of $C$ has the same parity as that of $D$. Thus, $n-2\ell_d\geq u+2$. Thus, $\overline{D}$ has the form:
\[
\overline{D} = \underbrace{9 \cdots 9}_u ~~\underbrace{\cdots}_{v}~~ \underbrace{x_D \cdots x_D}_{\geq u+2} ~~\underbrace{\cdots}_{u+v}
\]
for some $v\geq 0$ and some digit $x_D$.
Notice that when we form the rainbow pairs of $\overline{D}$, the value $x_D$ pairs with itself. By Lemma~\ref{inpairs}, the digits in a rainbow pair sum up to 8, 9, or 10, and at most one pair sums to 8 or 10. Thus, $x_D+x_D$ must be 8 or 10 (hence $x_D$ is either 4 or 5), and there is exactly one $\{x_D, x_D\}$ rainbow pair. So $\overline{D}$ and $\overline{C}$ have the forms  
\[
\overline{D} = \underbrace{9 \cdots 9}_u ~~\underbrace{\cdots}_{v}~~ \underbrace{x_D \cdots x_D}_{u+2} ~~\underbrace{\cdots}_{u+v} \text{ and} \;\;\; \overline{C} = \underbrace{9 \cdots 9}_{u+2} ~~\underbrace{\cdots}_{w}~~ \underbrace{x_C \cdots x_C}_{u} ~~\underbrace{\cdots}_{u+w+2}
\]
for some digit $x_C$ and where $v\geq0$ and $w\geq0$. The integers $C$ and $D$ have the same number of digits, so $v=w+1$. 

Suppose $x_D=4$. Then $u=0,$ since we can have at most one $\{4,4\}$ rainbow pair, and 4 cannot form a rainbow pair with any digit smaller than 4. Then, $D$ must also have a symmetric $\{4,4\}$ pair by Lemma~\ref{inpairs}. And it will be the innermost symmetric pair (Equation~\ref{D}). Since $\overline{C}-\underline{C}$ has a pair 4's in the center, it follows that:
$
\overline{C} = 99 ~~\underbrace{\cdots}_{w-1}~(b+5) ~(b)\underbrace{\cdots}_{w+1}
$
where $b\leq 4$. On the other hand, if we form the rainbow pairs in $\overline{C}$, notice that $b$ pairs with a digit less than or equal to $b$. Again, since the digits in rainbow pairs must sum to 8, 9, or 10 and since $b\geq 4$, this forces $b=4$ and $b+5=9$. The digit $b=4$ must form a rainbow pair with another 4, and therefore by Lemma~\ref{inpairs}, there is a rainbow pair that sums to 10, which by necessity must be $\{9,1\}$, since the only digit higher than $b=4$ in $\overline{C}$ is 9. All the remaining 9's in rainbow pairs must be paired with 0, since the sum of the digits must be 9. To summarize:
$\overline{C} = 99 ~~\underbrace{9\cdots 9}_{w-1}9 ~ 4~ 4 ~ 1 \underbrace{0\cdots 0}_{w-1}.
$
Calculating $\overline{C}-\underline{C}$, we have $D=\underbrace{9\cdots 9}_{w-1} 854441\underbrace{0\cdots 0}_{w-2} 1$, but then $\overline{D}-\underline{D}$ contains the digit 8, whereas $C$ does not, a contradiction.

We are left with the possibility that $x_D=5$ which implies $D$ has a rainbow pair involving two 5's. By Lemma~\ref{inpairs}, $D$ also has a $\{5,5\}$ symmetric pair and the pair must be the outermost digits of $D$ (Equation~\ref{D}). Since $\overline{C}$ starts with a 9, $\underline{C}$ must start with a 4 in order for the first digit of $D$ to be 5. Thus the smallest digit in $C$ is a 4. We have:
$
\overline{C} = \underbrace{9\cdots 9}_{u+2} ~y_{u+3}~\underbrace{\cdots}_{w-1}~\underbrace{x_C\cdots x_C}_{u}\underbrace{\cdots}_{u+w+1}  4
$.
Notice that $y_{u+3}$ forms a rainbow pair with 4, so $y_{u+3}$ must be a 4, 5, or 6. 
We know $y_{u+3}\neq 4$, since $y_{u+3}=4$ would force all the digits after $y_{u+3}$ to be 4, but $C$ must have more than two distinct digits (Lemma~\ref{morethan2}).  

Suppose $y_{u+3}=5$. In this case, it is impossible to have both a rainbow pair summing to 8 and another to 10, so all rainbow pairs must sum to 9.
\[
\overline{C} = \underbrace{9\cdots 9}_{u+2}~ \underbrace{5\cdots 5}_{u+w+1} ~ \underbrace{4\cdots 4}_{u+w+1}
\]
Then, calculating $\overline{C}-\underline{C}$ we have three outcomes, depending on the value of $w$.
$$
\overline{C}-\underline{C} = 
\begin{cases}
    \underbrace{5\cdots 5}_{u+1}~3~\underbrace{9\cdots 9}_{u}~5~\underbrace{4\cdots 4}_{u}~5 & \text{ if } w = 0\\
    \underbrace{5\cdots 5}_{u+1}~4~\underbrace{9\cdots 9}_{u+2}\underbrace{4\cdots 4}_{u+1}~5 & \text{ if } w = 1\\
    \underbrace{5\cdots 5}_{u+2}~\underbrace{1\cdots 1}_{w-2}~0~\underbrace{9\cdots 9}_{u+2}\underbrace{8\cdots 8}_{w-1}\underbrace{4\cdots 4}_{u+1}~5 & \text{ if } w \geq 2\\
\end{cases}
$$
If $w=0$, then the largest digit of $D$ is 9, and $\deltaD{1}=9-3=6$, which is not a digit in $C$, a contradiction. 
When $w=1$, then $C=D$, and we have a Kaprekar constant, rather than a 2-cycle. 
Finally, if $w\geq2$, then $10-\deltaD{1}=10-9=1$, but there is no digit equal to 1 in $C$. Thus, $y_{u+3}\neq 5$.

The final possibility is $y_{u+3}=6$. Since this digit forms a rainbow pair with sum 10, $\overline{C}$ has a pair that sums to 8, as well. This rainbow pair must be formed by digits between 6 and 4, so it consists of two 4's. Thus, 
$
\overline{C}= \underbrace{9\cdots 9}_{u+2} 6 \underbrace{5\cdots 5}_{u+w-1} ~ 4 ~ 4 ~ \underbrace{4\cdots 4}_{u+w-1} ~4 
$
By Lemma~\ref{inpairs}, since we have a rainbow pair with two 4's in $\overline{C}$, we must have a symmetric pair in $C$ with two 4's, as well. As we argued before, this implies $\overline{D}$ has a corresponding pair $b$ and $b+5$. Again, as argued to before, $b=4$, so then 
\[
\overline{D}= \underbrace{9\cdots 9}_{u+w+1}~ \underbrace{5\cdots 5}_{u+2} ~4 \underbrace{\cdots}_{u+w}  
\]
Because $\overline{C}$ has a $\{6,4\}$ rainbow pair, $C$ also has a symmetric pair $\{6,4\}$. Thus, $C$ either begins with 6 and ends with 4 or vice versa. Note $\overline{D}$ begins with a 9, so the smallest digit in $D$ is 3 or 5, in order for the first digit of $C$ to be 4 or 6. Since $D$ contains the digit 4, we conclude that the smallest digit of $D$ is 3. This implies $\overline{D}$ has a $\{9,3\}$ rainbow pair and hence a $\{9,3\}$ symmetric pair. But that is impossible because the sum of the pair is 12 (Equation~\ref{symmetricpairs}).  Thus, $y_{u+3}\neq 6$, and so there is no possible value for $y_{u+3}$ which implies $x_D\neq 5$. 

Thus, we conclude that $\ell_c = \ell_d$, and a Kaprekar 2-cycle $\{C,D\}$ with different digit sums does not exist other than $\{53955, 59994\}$.
\end{proof}

The following corollary follows directly from the proof of Theorem~\ref{sum}. In light of Corollary~\ref{ell}, we write $\ell$ in lieu of $\ell_c$ and $\ell_d$ moving forward. 
 \begin{corollary}\label{ell}
     Except for $\{53955, 59994\}$, every Kaprekar 2-cycle $\{C,D\}$ satisfies $\ell_c = \ell_d$. 
 \end{corollary}

\section{Kaprekar 2-cycles and the digit 9}~\label{no9+samesmallest}

In this section, we prove that both numbers in a Kaprekar 2-cycle must contain the digit 9. Note that this is not the case for Kaprekar constants. Numbers of the form $76_{x+1}43_x1$ for $x\geq 0$ form an infinite family of Kaprekar constants which do not contain the digit 9 (see Family (1) in Theorem~\ref{dolan}). 
We begin by proving a weaker result as follows.

\begin{theorem}\label{9inboth}
    Let $\{C,D\}$ be a Kaprekar 2-cycle. If one integer (say, $C$) contains the digit 9, then so does $D$.
\end{theorem}
\begin{proof}
Suppose that $C$ contains the digit 9 and that $D$ does not. Because $D$ does not contain the digit 9, $n=2\ell$ (considering Equation~\ref{D}). 
Because $C$ contains the digit 9 and $\overline{D} - \underline{D} = C$, it must be the case that one of the following is true (considering Equation~\ref{C}):
\begin{enumerate}
    \item $\deltaD{i}=z_i - z_{n+1-i} = 9$ for some $1\leq i<\ell$
    \item $\deltaD{\ell}-1=z_\ell - z_{\ell+1}-1 = 9$
    \item $9-\deltaD{i}=9 - (z_i - z_{n+1-i}) = 9$ for some $1<i\leq \ell$
    \item $10-\deltaD{1}=10 - (z_1 - z_{n}) = 9$
\end{enumerate}
Option (1) implies $z_i = 9$ and $z_{n+1-i} = 0$, which is not the case since $D$ has no 9's. Option (2) is not possible because each $z_i$ represents a single digit. Option (3) is not possible since $z_i > z_{n+1-i}$. In the final option,  $z_1= z_n+ 1$ implies that $D$ contains only two digits: $z_1$ and $z_1 - 1$, which contradicts Lemma~\ref{morethan2}. 
\end{proof}

With the above result, it remains to show that Kaprekar 2-cycles where neither $C$ nor $D$ contains the digit 9 do not exist. 
Similarly to~\cite{dolan}, we make the following additional definitions.

\begin{definition}\label{PQ}
    Let $\PC{i}$ denote the number of corresponding pairs in $\overline{C}$ with difference $i$. That is,
    \[\PC{i} = |\{\deltaC{j} : \deltaC{j} = i\}|\]
    In addition, define
\begin{equation*}
   \pC{1}(i)= \begin{cases}
        1 & \text{if } \deltaC{1}=i \\
        0 & \text{otherwise} \\
    \end{cases} \hspace{1cm}
      \pC{\ell}(i)= \begin{cases}
        1 & \text{if } \deltaC{\ell}=i \\
        0 & \text{otherwise} 
    \end{cases} 
\end{equation*}
and denote the sum of these values by $\pC{}(i)= \pC{1}(i) + \pC{\ell}(i)$.

Similarly, define 
$\PD{i} = |\{\deltaD{j} : \deltaD{j} = i\}|$. Define 
$\pD{1}(i)$ to be 1 or 0 according to whether or not $\deltaD{1}=i$, and similarly for $\pD{\ell}(i)$. Finally, define $\pD{}(i)= \pD{1}(i) + \pD{\ell}(i)$.
\end{definition}

\begin{lemma}\label{details}
   Let $\{C, D\}$ be a Kaprekar 2-cycle.   
    \begin{enumerate}
        \item There exist exactly two values of $i$ such that $\pC{}(i) = 1$. For all other values of $i$, $\pC{}(i) = 0$. Similarly for $\pD{}(i)$.
        \item If neither $C$ nor $D$ contains the digit 9, then $\pC{}(5) = 0$ and $\pD{}(5)=0$.
    \end{enumerate} 
    \end{lemma}
\begin{proof}
  Each of the functions $\pC{1}(i)$ and $\pC{\ell}(i),$ are equal to 1 for one value of $i$ and are 0 elsewhere. Thus, $\pC{}(i) = \pC{1}(i) + \pC{\ell}(i)$ can be nonzero for 1 or 2 values of $i$. If $\pC{}(i)$ is nonzero for exactly one value of $i$, then $\deltaC{1} = \deltaC{\ell}=i.$ This contradicts Lemma~\ref{notallsame}. So $\pC{}(i)$ is nonzero for exactly two values of $i$ and equals 1 at those values. This proves (1).

 Now consider statement (2). Suppose neither $C$ nor $D$ contains the digit 9, and suppose for contradiction that $\pC{}(5) = 1$. Then either $\deltaC{1}=5$ or $\deltaC{\ell}=5$. Suppose $\deltaC{1}=5$. Since $D$ begins with digit $\deltaC{1}$ and ends with digit $10-\deltaC{1}$, it follows that $D$ has a symmetric pair $\{5,5\}$. 
     By Lemma~\ref{inpairs}, since $p=q=5$, the symmetric pairs of $D$ and the rainbow pairs of $\overline{D}$ coincide. Therefore, there is a rainbow pair of 5's in $\overline{D}$. Because there are no 9's in $D$, rainbow pairs coincide with corresponding pairs of $\overline{D}$. However, there cannot be corresponding pairs of a digit with itself, so we have a contradiction. If $\deltaC{\ell} = 5,$ we reach a similar contradiction with a pair of 4's. Thus $\pC{}(5) = 0$.   A parallel argument establishes $\pD{}(5) = 0$.
\end{proof}

\begin{lemma}\label{not10}
Let $\{C, D\}$ be a Kaprekar 2-cycle in which neither $C$ nor $D$ contains the digit 9. Then the symmetric pairs of $C$ and corresponding pairs of $\overline{C}$ coincide. Similarly for $D$.
\end{lemma}
\begin{proof}
We consider two cases corresponding to whether $C$ contains the smallest digit in the 2-cycle. First, suppose that $C$ contains the smallest digit. 
By Theorem~\ref{smallest}, $y_n=\deltaD{\ell}-1.$ Since $y_n\leq z_n$, we observe:
$$\deltaD{1}+\deltaD{\ell} = z_1-z_n+y_n+1 \leq 8+(y_n-z_n)+1 = 9 +(y_n-z_n)\leq 9.$$
Since $\deltaD{1}+\deltaD{\ell} \neq 10$, Lemma~\ref{inpairs} implies that the symmetric pairs of $C$ and rainbow pairs of $\overline{C}$ coincide. Since $C$ does not contain the digit 9, the rainbow pairs and corresponding pairs of $\overline{C}$ coincide, and the result follows.

Now suppose that $C$ does {\em not} contain the smallest digit in the Kaprekar 2-cycle. Thus $D$ contains the smallest digit, and it occurs in $D$ exactly once as the $\ell^{th}$ digit of $D$: $z_n = \deltaC{\ell}-1$ (Theorem~\ref{smallest}). As we have argued above, because $D$ contains the smallest digit, the corresponding pairs of $\overline{D}$ and symmetric pairs of $D$ coincide.  The symmetric pair $\{\deltaC{\ell}-1, 9-\deltaC{\ell}\}$ of $D$ contains the smallest digit in the 2-cycle and the smallest digit appears only once, so this symmetric pair is the outermost corresponding pair of $\overline{D}$. That is, $z_1=9-\deltaC{\ell}$ and $z_n = \deltaC{\ell}-1.$ We conclude $z_1 = 8-z_n$ and $\deltaD{1} = z_1-z_n= 8-2z_n$.

Suppose for contradiction that the rainbow pairs of $\overline{C}$ and symmetric pairs of $C$ do not coincide. 
By Lemma~\ref{inpairs}, the symmetric pairs $\{\deltaD{1}, 10-\deltaD{1}\}$ and $\{\deltaD{\ell}-1, 9-\deltaD{\ell}\}$ 
of $C$ switch to form rainbow (and, hence, corresponding) pairs in $\overline{C}$ as follows: $\{\deltaD{1}, \deltaD{\ell}-1\}$ and $\{10-\deltaD{1}, 9-\deltaD{\ell}\}$, for the only other possible pairing implies $\deltaD{1}=\deltaD{\ell}$, contradicting Lemma~\ref{notallsame}. Also by Lemma~\ref{inpairs}, $\deltaD{1}+\deltaD{\ell}=10$, and so since $\deltaD{1}= 8-2z_n$, we have $\deltaD{\ell}=2+2z_n.$ We write corresponding pairs in $\overline{C}$ in terms of $z_n$: $\{8-2z_n, 1+2z_n\}$ and $\{2+2z_n, 7-2z_n\}$.

Since  $\deltaD{1}>\deltaD{\ell}$, observe that $z_n = 0$ or 1.  Suppose $z_n=1$. Then $\overline{C}$ has a corresponding pair equal to $\{4,5\}$. It follows that the innermost corresponding pair (that is, the $\ell^{th}$ pair) of $\overline{C}$ must be $y_\ell = 5$ and $y_{\ell+1}= 4$. Then, $D$ contains digit $\deltaC{\ell}-1=y_\ell - y_{\ell+1}-1=0$, contradicting the fact that the smallest digit in the 2-cycle is $z_n=1$. Therefore, $z_n\neq 1$.

Suppose $z_n=0$. Then $\overline{C}$ has a corresponding pairs $\{1,8\}$ and $\{2,7\}$, and hence $\deltaC{1}=7$ and $\deltaC{j}=7-2=5$ for some $1<j\leq\ell$. Since $D$ contains the digit $z_1=8-z_n=8$, one of the expressions involving $\deltaC{i}$ in Equation~\ref{D} must equal 8. Thus, there must exist an $i$ such that $\deltaC{i}=1$ (for then, $9-\deltaC{i}=8$). The existence of a corresponding pair in $C$ with difference 1 guarantees that the pair $\{2, 7\}$ is {\em not} the innermost corresponding pair of $\overline{C}$, and $D=\overline{C}-\underline{C}$ contains a symmetric pair $\{\deltaC{j}, 9-\deltaC{j}\}=\{5, 4\}$. Recall that symmetric pairs of $D$ are also corresponding pairs of $\overline{D}$. It follows that the innermost corresponding pair of $\overline{D}$ must be $z_\ell = 5$ and $z_{\ell+1}= 4$. Then, $C$ contains digit $\deltaD{\ell}-1=z_\ell - z_{\ell+1}-1=0$, contradicting the fact that $C$ does not contain the digit 0. So, $z_n\neq 0$. 

Therefore, the rainbow pairs of $\overline{C}$ and symmetric pairs of $C$ coincide. Again, since $C$ does not contain the digit 9, the rainbow pairs and corresponding pairs of $\overline{C}$ coincide, and the result follows.
\end{proof}

\begin{lemma}\label{equations}
    Let $\{C, D\}$ be a Kaprekar 2-cycle where neither integer contains the digit 9. The functions $\PD{i}$, $\PC{i}$,
    and $\pC{}(i)$ are connected to each other as follows.
    
{\small\begin{tabular}{lllll}
        ($a$) & $\PD{1} = \PC{4} + \PC{5} - \pC{}(4)$            &  & ($e$) & $\PD{2} = \pC{}(4) + \pC{}(6)$ \\
        ($b$) & $\PD{3} = \PC{3} + \PC{6} - \pC{}(3) - \pC{}(6)$ &  & ($f$) & $\PD{4} = \pC{}(3) + \pC{}(7)$ \\
        ($c$) & $\PD{5} = \PC{2} + \PC{7} - \pC{}(2) - \pC{}(7)$ &  & ($g$) & $\PD{6} = \pC{}(2) + \pC{}(8)$ \\
        ($d$) & $\PD{7} = \PC{1} + \PC{8} - \pC{}(1) - \pC{}(8)$ &  & ($h$) & $\PD{8} = \pC{}(1)$
    \end{tabular}}
\end{lemma}
\begin{remark}\label{remark}
    Parallel equations to (a) through (h) hold, with the roles of $\PD{i}$,  $\PC{i}$ and $\pC{}(i)$ interchanged with $\PC{i}$, $\PD{i}$, and $\pC{}(i)$, respectively, with a parallel argument.
\end{remark}
\begin{proof}
Lemma~\ref{not10} implies that we may compute $\PD{i}$ considering the symmetric pairs of $D$ rather than the corresponding pairs of $\overline{D}$. The differences of the symmetric pairs in $D$ are the following values: 
$$|10-2\deltaC{1}|, |9-2\deltaC{2}|, \ldots, |9-2\deltaC{\ell-1}|, |10-2\deltaC{\ell}|.$$  
Observe that the first and last differences in the above list are even integers while the remaining are odd. Therefore, to calculate, say, $\PD{5}$, we need to find all $\deltaC{i}$ where $1<i<\ell$ such that $|9-2\deltaC{i}| = 5$. Notice that $|9-2\deltaC{i}| = 5$ if and only if $\deltaC{i} = 2$ or 7. Therefore $\PD{5} = \PC{2}-\pC{}(2)+\PC{7}-\pC{}(7)$, as desired. Equations (b) and (d) similarly follow. For Equation (a), we similarly conclude $\PD{1} = \PC{4}-\pC{}(4)+\PC{5}-\pC{}(5)$, but  Lemma~\ref{details} implies $\pC{}(5)=0$.

Now we consider the value of $\PD{i}$ on even integers. Say, $\PD{2}$. We must check whether $|10-2\deltaC{1}| = 2$ or $|10-2\deltaC{\ell}| = 2$. The first equation (respectively, second equation) holds exactly when $\deltaC{1}$ (respectively, $\deltaC{\ell}$) equals 4 or 6. Therefore, $\PD{2} = \pC{}(4) + \pC{}(6).$ Equations (f), (g), and (h) follow in a similar way. In (h), it should be noted that since $C$ does not contain the digit 9, $\pC{}(9)=0$.
\end{proof}

\begin{lemma}\label{Peven}
    Let $\{C, D\}$ be a Kaprekar 2-cycle where neither integer contains the digit 9. The function $\PD{i}$ satisfies the following: $\PD{2i} \in \{0,1\}$, and $\PD{2i} = 1$ for exactly two values of $i$. 
\end{lemma}
\begin{proof}
    By Lemma~\ref{details} and Equations (e), (f), (g), and (h) of Lemma~\ref{equations}, we observe
      \[
      \PD{2} +\PD{4} + \PD{6} + \PD{8} =\pC{}(1) + \pC{}(2)+\pC{}(3)+\pC{}(4)+\pC{}(6)+\pC{}(7)+\pC{}(8)=2
      \]
  Thus $\PD{2i}\in\{0,1,2\}$ for all $i$. Suppose that $\PD{2i} = 2$ for some $i$. This would imply that $\pC{}(4)=\pC{}(6)= 1$ or $\pC{}(3) =\pC{}(7)= 1$ or $\pC{}(2) =\pC{}(8)= 1$. This, in turn, would imply that $\deltaC{1}+\deltaC{\ell} = 10$. By Lemma~\ref{inpairs}(2), the symmetric pairs of $D$ and rainbow pairs of $\overline{D}$ do not coincide. This contradicts Lemma~\ref{not10}. Therefore, $\PD{2i}=1$ for exactly two values of $i$ and otherwise $\PD{2i}=0$.
\end{proof}

\begin{theorem}\label{no2cycles}
    In a Kaprekar 2-cycle $\{C, D\}$, both $C$ and $D$ contain the digit 9. 
\end{theorem}
\begin{proof}
    By Theorem~\ref{9inboth}, it suffices to show that a Kaprekar 2-cycle $\{C,D\}$ such that neither $C$ nor $D$ contains the digit 9 does not exist. Suppose such a Kaprekar 2-cycle exists. By Lemma~\ref{not10}, the symmetric pairs of $C$ and the corresponding pairs of $\overline{C}$ coincide. Similarly for $D$.  Equations (a) through (h) from Lemma~\ref{equations} apply, as well as parallel formulas as follows (see Remark~\ref{remark}):
 
 {\small\begin{tabular}{lllll}
    ($a'$) & $\PC{1} = \PD{4} + \PD{5} - \pD{}(4)$ & & ($e'$) & $\PC{2} = \pD{}(4) + \pD{}(6)$ \\
    ($b'$) & $\PC{3} = \PD{3} + \PD{6} - \pD{}(3) - \pD{}(6)$ && ($f'$) & $\PC{4} = \pD{}(3) + \pD{}(7)$ \\
    ($c'$) & $\PC{5} = \PD{2} + \PD{7} - \pD{}(2) - \pD{}(7)$ && ($g'$) & $\PC{6} = \pD{}(2) + \pD{}(8)$ \\
    ($d'$) & $\PC{7} = \PD{1} + \PD{8} - \pD{}(1) - \pD{}(8)$ && ($h'$) & $\PC{8} = \pD{}(1)$
\end{tabular}}

    Without loss of generality, suppose that $D$ has the smallest digit in the 2-cycle (this digit may or may not be a digit in $C$, as well). Then by Theorem~\ref{smallest}, the smallest digit (which is $z_n$) must be in the $\ell^{th}$ position of $D$ and so $z_n = \deltaC{\ell}-1$. Since $\{z_1, z_n\}$ form a corresponding pair in $\overline{D}$, the two digits must also form a symmetric pair in $D$. Therefore, $z_1 = 9-\deltaC{\ell}$ and $\deltaD{1} = z_1-z_n = 10-2\deltaC{\ell}$. Because $\deltaC{\ell}-1$ is the smallest digit in the 2-cycle and $\deltaD{\ell}-1$ is a digit in $C$, we conclude $\deltaC{\ell}\leq \deltaD{\ell}$. 

    First, we claim that $\deltaC{\ell}\geq 2$.  Suppose for contradiction that $\deltaC{\ell}=1$. We split into two cases: (1) $\deltaD{\ell}=\deltaC{\ell}= 1$ and (2) $\deltaD{\ell}>\deltaC{\ell}=1$. Notice that in either case, $\deltaD{1}=10-2\deltaC{\ell}=8$, which implies $\pC{}(1)=\pD{}(8)=\PD{8}=1$.

    {\bf Case 1:} $\deltaD{\ell}=\deltaC{\ell}= 1$. Then $\deltaC{1}= 8,$ for similar reasons as above. So then $\pD{}(1)=\pC{}(8)=1$.  Because $\pC{}(i)$ and $\pD{}(i)$ are  nonzero for exactly two values of $i$ (Lemma~\ref{details}), we conclude that the functions are zero for $i=2, 3, 4, 5, 6, 7.$
Combining Equations (b), (g), (b$'$), and (g$'$) from Lemma~\ref{equations}, we have
\begin{equation}\label{pqEquation1}
    \pD{}(3) + \pD{}(6) + \pC{}(3) + \pC{}(6) = \pC{}(2) + \pC{}(8)+ \pD{}(2)+\pD{}(8)
\end{equation}
Plugging the known values of $\pC{}(i)$ and $\pD{}(i)$ into Equation~\ref{pqEquation1}, the equation becomes $0 = 2$, a contradiction. 

     {\bf Case 2:} $ \deltaD{\ell}>\deltaC{\ell}=1$.  Then $\PD{1}=\pD{}(1)=0$.
Equations (g$'$) and (d$'$) yield $\PC{6} =\pD{}(2) + \pD{}(8)>0$ and 
$\PC{7} = \PD{1} + \PD{8} - \pD{}(1) - \pD{}(8) = 0.$ Therefore $\deltaC{1}=6$ and $\pC{}(6)=1$. Since $\pC{}(i)$ is nonzero for only two values of $i$ and $\pC{}(1)=\pC{}(6)=1$, we know that all other values of $\pC{}(i)$ are 0. Because $\PD{2}= \pC{}(4)+\pC{}(6)=1$ and $\PD{1}=0,$ we conclude $\deltaD{\ell}=2$. It follows that $\deltaD{1} + \deltaD{\ell}=10$, a contradiction (Lemma~\ref{inpairs} and Lemma~\ref{not10}). 

Therefore, we conclude  $\deltaC{\ell}\geq 2$, as claimed. Since $\deltaD{\ell}\geq\deltaC{\ell}$, we also know $\deltaD{\ell}\geq2$, as well. Thus $\PC{1} =\PD{1} = 0$, which, in turn, implies $\pC{}(1)$, $\pD{}(1)$, $\PC{8}$, $\pC{}(8)$, $\PD{8}$, and $\pD{}(8)$ are all zero.

     Plugging these values into Equation (a) in Lemma~\ref{equations}, we find $\PC{5} = \pC{}(4)-\PC{4}\leq 0$, which implies that $\PC{5} = 0$. From Equation (d), we have $\PD{7} = 0$, and hence $\pD{}(7) = 0$. Similarly, from Equations (a$'$) and (d$'$) we find $\PD{5}$, $\PC{7}$, and $\pC{}(7)$ are all zero.

Because  $\PD{8} =\PD{7}= 0$ and $\PD{6}\in\{0,1\}$, it follows that $\PD{6} = \pD{}(6).$ Similarly, we conclude 
that $\PC{6} = \pC{}(6).$ Plugging these expressions into Equations (b) and (b$'$), we find
\[
 \PD{3} =\PC{3} - \pC{}(3) 
 \;\;\text{ and }\;\;    
 \PC{3} =\PD{3} - \pD{}(3).
\]
Combining the above two equations together, we conclude that $\pD{}(3) = -\pC{}(3)$, which forces $\pD{}(3)=\pC{}(3)=0$.

The fact that $\pD{}(3)$ and $\pC{}(3)$ both equal zero implies (via Equations (f) and (f$'$)) that $\PD{4}$ and $\PC{4}$ are both 0 and hence so are $\pD{}(4)$ and $\pC{}(4)$. 
Lemma~\ref{Peven} guarantees that $\pD{}(i)$ and $\pC{}(i)$ are nonzero for exactly two values of $i$. By process of elimination, we conclude that $\pD{}(2)$, $\pD{}(6)$, $\pC{}(2)$, and $\pC{}(6)$ all equal 1.

 Moreover, from Equations (e), (g), (e$'$), and (g$'$) we conclude that $\PD{2}$, $\PD{6}$, $\PC{2}$, and $\PC{6}$ all equal 1, so there is exactly one corresponding pair with difference 6 in $\overline{C}$ and $\overline{D}$, and exactly one corresponding pair in both $\overline{C}$ and $\overline{D}$ with difference 2. Equation (b) simplifies to $\PD{3} = \PC{3}$, so $\overline{C}$ and $\overline{D}$ have the same number (call it $x$) of corresponding pairs with difference 3. 

 Recall that the corresponding pairs of $\overline{C}$ and symmetric pairs of $C$ coincide, and the same for the pairs of $\overline{D}$ and $D$ (Lemma~\ref{not10}). Thus, the corresponding pairs of digits in $C$ and $D$ also have the property that they sum to 8, 9, or 10 (with exactly one pair summing to 8 and exactly one pair summing to 10). The only pair of digits that has difference 3 with a sum of 8, 9, or 10 is $\{6,3\}$. The pairs of digits with difference 6 that meet the criteria are $\{2,8\}$ and $\{1,7\}$. Finally, the pairs with difference 2 that meet the criteria are $\{6,4\}$ and $\{5,3\}$. Putting this together, there are two possibilities for $C$ and $D$, as follows where we take $x\geq 0$:
 \[
 86_x53_{x+1}2 
 \;\;\; \text{ or }\;\;\; 
 76_{x+1}43_x1
 \]
Applying the Kaprekar process to the number on the left yields the number on the right (after ordering the digits from largest to smallest). The number on the right is a Kaprekar constant (see Family (1) in Theorem~\ref{dolan}). Therefore, neither of these numbers leads to a Kaprekar 2-cycle. 
    \end{proof}

\section{Kaprekar 2-cycles with the same smallest digit in each number}\label{9+samesmallest}

In this section, we classify Kaprekar 2-cycles which have the same smallest digit in each number, proving Theorem~\ref{familyA}. Our overall strategy is to define two matrices associated with such a 2-cycle. We find constraints on the entries of these matrices. These, in turn, constrain the Kaprekar 2-cycle.

\begin{lemma}\label{pairs}
  Suppose that $\{C,D\}$ is a Kaprekar 2-cycle such that both $C$ and $D$ contain the same smallest digit. Then $\deltaD{1} + \deltaD{\ell} = 10$, $\deltaC{1}=\deltaD{1}$, and $\deltaC{\ell}=\deltaD{\ell}$. Moreover, if we let $z=n-2\ell$, then $\overline{C}$ and $\overline{D}$ have the form:
\begin{equation}\label{CD}
    \overline{C} = 9_{a+z}8_b7_c6_d5_e4_e3_d2_c1_b0_a ~~~\text{ and } ~~~
\overline{D} = 9_{a+z}8_f7_g6_h5_i4_i3_h2_g1_f0_a,
\end{equation}
for some nonnegative integers $a, b, c, d, e, f, g, h,$ and $i$.
\end{lemma}
\begin{proof}
Theorem~\ref{smallest} implies that $\deltaD{\ell}-1$ is the smallest digit of $C$. Since  $C$ and $D$ have the same the smallest digit,  $\deltaD{\ell}-1 = z_n$. Since $D$ contains the digit 9 (Theorem~\ref{no2cycles}),
$$  \deltaD{1} + \deltaD{\ell}-1  = (z_1 - z_n) + z_n = z_1 = 9,
$$  
Therefore $\deltaD{1} + \deltaD{\ell}= 10$. Using a parallel argument, $\deltaC{1} + \deltaC{\ell} = 10$, as well.

Because $C$ and $D$ have the same largest digit and same smallest digit, $\deltaC{1}=\deltaD{1}$. Moreover, since $\deltaC{1}+\deltaC{\ell}=\deltaD{1}+\deltaD{\ell} = 10$, we conclude $\deltaC{\ell}=\deltaD{\ell}$.

Because $\deltaD{1}+\deltaD{\ell} = 10$, Lemma~\ref{inpairs}(2) implies that $2\ell$ digits of $C$ belong in pairs, all of which sum to 9. The remaining $z=n-2\ell$ digits of $C$ are 9's. Therefore $\overline{C}$ can be expressed as in Equation~\ref{CD}. The digit 0 occurs the same number of times (say, $a$ times) in both $C$ and $D$, as it is the smallest possible digit (Theorem~\ref{smallest}).  Beyond that, the construction of $\overline{D}$ is similar and results in Equation~\ref{CD}. 
\end{proof}

The goal of this section is to show that the integers $a, b, c, d, e, f, g, h, i,$ and $z$ in Equation~\ref{CD} must have values of the form given in Theorem~\ref{familyA}. To that end, we define a pair of matrices $M_C$ and $M_D$ associated to the Kaprekar 2-cycle $\{C, D\}$.
\begin{definition}
     For a Kaprekar 2-cycle $\{C, D\}$, let $M_C$ be a $5\times 5$ matrix where the $(i,j)^{th}$ entry is the number of corresponding pairs of $\overline{C}$ of the form $\{i-1, 10-j\}$, where $1\leq i,j\leq 5$. The matrix $M_D$ is defined similarly. 
\end{definition}
The matrix entries of $M_C$ and $M_D$ satisfy several relationships. These relationships parallel those found by Dolan in Lemma 5 of~\cite{dolan} for Kaprekar constants. 

\begin{lemma}\label{matrix}
    Suppose that $\{C,D\}$ is a Kaprekar 2-cycle such that both $C$ and $D$ contain the same smallest digit. With the notation for $\overline{C}$ and $\overline{D}$ as in Equation~\ref{CD}, 
    the associated matrices $M_C$ and $M_D$ are as below. The matrix entries satisfy (1) through (4) below.
    \[ 
    M_C = 
    \begin{bmatrix} 
        a & 0 & 0 & 0 & 0\\ 
        f_1 & g_2 & 0 & 0 & 0\\ 
        g_1 & h_2 & i_3 & 0 & 0\\ 
        h_1 & i_2 & i_5 & h_4 & 0\\ 
        i_1 & i_4 & h_3 & g_3 & f_2
        \end{bmatrix}   
    \;\; 
    M_D = 
    \begin{bmatrix} 
        a & 0 & 0 & 0 & 0\\ 
        b_1 & c_2 & 0 & 0 & 0\\ 
        c_1 & d_2 & e_3 & 0 & 0\\ 
        d_1 & e_2 & e_5 & d_4 & 0\\ 
        e_1 & e_4 & d_3 & c_3 & b_2
    \end{bmatrix}
    \]
    \begin{enumerate}
        \item The row sums of $M_C$ are $a, b, c, d, e$, respectively. The row sums of $M_D$ are $a, f, g, h, i$, respectively. 
        \item If $e>0$, then the column sums of $M_C$ are $a+z, b, c, d,$ and $e-z$, respectively. Similarly, if $i>0$, then the column sums of $M_D$ are $a+z, f, g, h,$ and $i-z$, respectively. 
        \item $b = \sum b_j$. Similarly for $c, d, e, f, g, h, i$.
        \item For both matrices, if $m_{ij}>0,$ then for all $p$, $q$ such that $(p-i)(q-j)<0$, we have $m_{pq} = 0$.
    \end{enumerate}
\end{lemma}
\begin{proof} 
We prove the statements for $M_C$, as the arguments for $M_D$ are parallel. Because $\overline{C}=9_{a+z}8_b7_c6_d5_e4_e3_d2_c1_b0_a$, when forming the corresponding pairs in $\overline{C}$, the digit 0 pairs only with the digit 9. So, the number of corresponding pairs $\{0,9\}$ is exactly $a$, and the top row of $M_C$ has the entries $a, 0, 0, 0, 0$, respectively. The sum of the first row's entries is $a$, as claimed. 

Again, considering the form of $\overline{C}$ given above, the digit 1 will possibly pair with the digits 9 or 8 when forming the corresponding pairs, but no other digits. Let $f_1$ denote the number of corresponding pairs of the form $\{9,1\}$, and let $g_2$ denote the number of corresponding pairs of the form $\{8,1\}.$ 
It follows that the entries of the second row are $f_1, g_2, 0, 0, 0$, respectively, and the sum of the second row of $M_C$ is $b$, as claimed. Similarly, the row sums are $c, d,$ and $e$, respectively, for the 3rd, 4th, and 5th rows of $M_C$.

Now consider statement (2) of the lemma. Suppose that $e>0$. Recall that the central $z=n-2\ell$ digits of $\overline{C}$ must all be equal, where $\ell = a+b+c+d + e$. Counting inward from the right side of $\overline{C}$, we observe that this center digit must be 5 since $e>0$. Hence $e\geq z$. It must be the case that $e-z$ of the 5's form a corresponding pair with 4, and every digit larger than 5 in $\overline{C}$ forms a corresponding pair with a digit less than or equal to 4. The sum of the first column of $M_C$ counts the number of corresponding pairs of the form $\{9,0\}$, $\{9,1\}$, $\{9,2\}$, $\{9,3\}$, and $\{9,4\}$, which accounts for every appearance of a 9 in $\overline{C}$, so this sum is $a+z$. A similar argument applies to columns 2, 3, and 4 of $M_C$. The final column sum is the number of $\{4,5\}$ corresponding pairs in $M_C$, of which there must be $e-z$.

Next, consider statement (3). The value $b$ represents the number of 1's and the number of 8's in $C$. 
Consider the first $\ell$ digits of $C$ in Equation~\ref{C}. The number of digits in this region equal to 1 is $\PD{1} - \pD{\ell}(1) +\pD{\ell}(2)$ (recalling Definition~\ref{PQ}). In the last $\ell$ digits of $C$, the number of digits equal to 1 is $\PD{8} - \pD{1}(8)+\pD{1}(9)$. Thus, we have:
\[
b= \PD{1}+ \PD{8} - \pD{\ell}(1) +\pD{\ell}(2)  - \pD{1}(8)+\pD{1}(9).
\]
Lemma~\ref{pairs} tells us that $\deltaD{1} + \deltaD{\ell} = 10$, which implies $\pD{1}(j) - \pD{\ell}(10-j) = 0$ for all $j$. Therefore, $b= \PD{1}+ \PD{8} = b_1 + b_2,$ as desired. A similar argument holds for $c, d, e, f, g, h,$ and $i$.

Finally, we consider statement (4) of the lemma. Suppose that two entries $m_{ij}$ and $m_{pq}$ in the matrix $M_C$ are both nonzero. This means that there exist corresponding pairs in $\overline{C}$ of the form $\{i-1,10-j\}$ and $\{p-1,10-q\}$. Note that $1\leq j \leq i\leq 5$ and $1\leq q \leq p\leq 5$ and hence the digits within each pair can be ordered as follows: $i-1\leq 10-j$ and $p-1\leq 10-q$. Because the corresponding pairs of $\overline{C}$ are nested, it must be the case that either
\[
i-1 \leq p-1\leq 10-q \leq 10-j 
\;\;\;\text{ or }\;\;\;
p-1\leq i-1\leq 10-j \leq 10-q.
\]
Regardless of which option holds, the signs of $p-i$ and $q-j$ coincide, forcing $(p-i)(q-j)\geq 0$. Therefore if $(p-i)(q-j)< 0$, it must be the case that one of $m_{ij}$ or $m_{pq}$ equals 0.
\end{proof}
\begin{lemma}\label{formulas}
The entries of the matrices $M_C$ and $M_D$ defined in Lemma~\ref{matrix} satisfy the following two equations:
\begin{equation}\label{1}
    e_3+i_3 = h_1+h_3+d_1+d_3+f_2+b_2 
\end{equation}
\begin{equation}\label{2}
        2h_3+2d_3 = i_2+e_2+g_1+c_1 
\end{equation}
\end{lemma}
\begin{proof}
These identities follow from considering the different ways to express the variables $a$ through $i$ using the relationships among the matrix entries discussed in Lemma~\ref{matrix}. For Equation~\ref{1}, we start by expressing $h$ and $i$ in different ways:
\begin{align*}
    h + i &= (d_1 + e_2 + e_5 + d_4) + (e_1 + e_4 + d_3 + c_3 + b_2)\\
    (d_4+c_3) + i &= (d_1 + e_2 + e_5 + d_4) + (e_1 + e_4 + d_3 + c_3 + b_2)\\
    i_1 + i_2+ i_3 + i_4+ i_5 &= (e_1 + e_2 + e_4 + e_5) + (d_1 + d_3) + b_2\\
    i_1 + i_2+ i_3 + i_4+ i_5 &= (i_1+i_4+h_3+g_3+f_2-e_3) + (d_1 + d_3) + b_2\\
    d - h_1 - h_4 + i_3  &= (h_3+g_3+f_2-e_3) + (d_1 + d_3) + b_2\\
    d_2+d_4+ i_3  &= h_1 +h_3 +d+ f_2-e_3 + b_2\\    
    e_3 + i_3  &= h_1 +h_3 +d_1+d_3+ f_2 + b_2 
\end{align*}
For Equation~\ref{2}, we begin by expressing the integers $h$ and $g$ in different ways.
\begin{align*}
h + g &= (d_1 + e_2 + e_5 + d_4) + (c_1 + d_2 + e_3)\\  
h + (e_3+e_5+d_3) &= (d_1 + e_2 + e_5 + d_4) + (c_1 + d_2 + e_3)\\  
h + d_3 &= (d_1 + d_2 + d_4) + e_2  +c_1\\  
h + d_3 &= (h_1+i_2+i_5+h_4-d_3) + e_2  +c_1\\  
h_2+h_3 + 2d_3 &= i_2+i_5 + e_2  +c_1\\  
h_3 + 2d_3 &= i_2+(i_5-h_2) + e_2  +c_1\\   
h_3 + 2d_3 &= i_2+(g_1 - h_3) + e_2  +c_1\\   
2h_3 + 2d_3 &= i_2+ e_2  +g_1 +c_1
\end{align*}
The final steps of the above calculation are justified by setting the two possible expressions for $c$ coming from the matrix $M_C$ equal to each other. We have $g_1 + h_2 + i_3 = i_3 + i_5 + h_3$, which simplifies to $i_5-h_2 = g_1 - h_3.$
\end{proof}

\begin{lemma}\label{45}
    If $\{C,D\}$ is a Kaprekar 2-cycle where $C$ and $D$ both contain the same smallest digit, then $C$ and $D$ both contain the digits 4 and 5. That is, $e>0$ and $i>0$. 
\end{lemma}
\begin{proof}
Without loss of generality, suppose for contradiction that $C$ has no 4's or 5's. By Lemma~\ref{pairs}, $\overline{C}=9_{a+z}8_b7_c6_d3_d2_c1_b0_a$ for some nonnegative integers $a, b, c, d,$ and $z$. Recall that the center $z$ digits of $C$ are the same digit. The $a+b+c+d$ digits on either side of this center block of size $z$ form corresponding pairs. 

Thus, the digits in a corresponding pair $\{y_i, y_{n+1-i}\}$ of $\overline{C}$ satisfy $y_i \in \{6, 7, 8, 9\}$ and $y_{n+1-i} \in \{0, 1,2, 3\}$. The values  $\deltaC{i}=y_i-y_{n+1-i}$ are thus restricted: $\deltaC{i} \in \{3,4,5,6,7,8,9\}$.
Recall from Lemma~\ref{pairs} that
$\deltaC{1}+\deltaC{\ell} = 10$. By Lemma~\ref{notallsame}, it is impossible for $\deltaC{1}=\deltaC{\ell}=5$. Thus, there are two possibilities: (1) $\deltaC{1}=6$ and $\deltaC{\ell}=4$ or (2) $\deltaC{1}=7$ and $\deltaC{\ell}=3$.

\textbf{Case 1: } Suppose $\deltaC{1}=6$ and $\deltaC{\ell}=4$. Because $y_1=9$ and $\deltaC{1}=y_1-y_n=6$, it follows that $y_n=3$. Thus $C$ has no digits equal to 0, 1, or 2, which implies $a=b=c=0$ and $\overline{C}=9_{z}6_d3_d$ (Equation~\ref{CD}). Because $\deltaC{\ell}=4$, there must be a corresponding pair in $\overline{C}$ with a difference of 4. There is no such pair, a contradiction.

\textbf{Case 2: } Suppose $\deltaC{1} =7$ and $\deltaC{\ell} = 3$. We have $\deltaC{1}=7 = y_1-y_n = 9-y_n$. So, $y_n=2$. Thus the smallest digit in both $C$ and $D$ is 2. The digit 2 (and, hence, also the digit 7) appears the same number of times in both $C$ and $D$ by Theorem~\ref{smallest}. Moreover, neither $C$ nor $D$ contains the digit 8, since the number of 8's equals the number of 1's, and so $\overline{C}=9_{z}7_c6_d3_d2_c$.

Since $\deltaC{\ell}=3$, there must be a corresponding pair in $\overline{C}$ with difference 3. Because $y_i \in \{6, 7, 9\}$ and $y_{n+1-i} \in \{2, 3\}$, this corresponding pair must be $\{3,6\}$. Collecting these observations, we have:
$\overline{C}=9_z7_c6_d3_d2_c$, where $z$, $c$, and $d$ are strictly positive. Moreover, $d>z$ for otherwise 6 and 3 will not form the innermost corresponding pair.  Also, $\overline{D}=9_z7_c6_h5_i4_i3_h2_c$.

By Lemma~\ref{pairs}, $\deltaD{1}=7$ and $\deltaD{\ell} = 3$. None of the remaining $\deltaD{i}$ are 4 or 5, since that would produce a digit of 4 or 5 in $C$. Thus, there are no corresponding pairs in $\overline{D}$ of the form $\{9,4\}$, $\{7,2\}$, or $\{7,3\}$. Considering that $\overline{D}=9_z7_c6_h5_i4_i3_h2_c$, and the fact that 3 and 2 cannot form corresponding pairs with 7, we are forced to conclude $z\geq c+h$.
Recall $\deltaD{\ell}=3$. Since $z\geq c+h$, only possible corresponding pair in $\overline{D}$ with difference 3 is $\{4,7\}$. Thus $i>0$. All digits between the $\ell^{th}$ corresponding pair $\{4,7\}$ must be equal, and there must be $z$ of them. Since $i>0$, this forces $h=0$ and $i=z$, and 
$\overline{D}=9_z7_c5_z4_z2_c$.
Because there are no $\{9,4\}$ corresponding pairs in $\overline{D}$, $z = c$, and
$\overline{D}=9_z7_z5_z4_z2_z$ and $\overline{C}=9_z7_z6_d3_d2_z$.
Since $C$ and $D$ have the same number of digits, 
$d=z$. This contradicts the fact that $d>z$. 

Therefore, $C$ must have a nonzero number of 4's and 5's.
\end{proof}

We are now poised to prove Theorem~\ref{familyA}, which we now restate.

\begin{manualtheorem}{1.5}    
If $\{C,D\}$ is a Kaprekar 2-cycle where $C$ and $D$ both have the same smallest digit, then $\{\overline{C},\overline{D}\}$ is
    $$ \{9_{z}8_{z-v-u}7_{z+v}6_{z-v}5_z4_z3_{z-v}2_{z+v}1_{z-v-u}, 9_{z}8_{z-v-u}7_{z+u}6_{z-u}5_z4_z3_{z-u}2_{z+u}1_{z-v-u} \} $$
    where $u$, $v$, and $z$ are nonnegative integers such that $z>u+v>0$ and $u\neq v.$
\end{manualtheorem}
\begin{proof}
Any Kaprekar 2-cycle $\{C, D\}$ fits into one of the following four cases (where we swap $C$ and $D$, if necessary) with regard to the entries of the matrices $M_C$ and $M_D$. We show that Case 1(a) produces the Kaprekar 2-cycles in the statement of this theorem, while all other cases are dead ends.
\begin{multicols}{2}
\noindent{ Case 1.} $e_3>0$
\begin{enumerate}
    \item[(a)] $g_2 = 0$ and $c_2=0$
    \item[(b)] $g_2 = 0$ and $c_2 > 0$ 
    \item[(c)] $g_2>0$ 
    \end{enumerate}
\columnbreak
\noindent{ Case 2.} $e_3=0$ and $i_3=0$
\end{multicols}

\noindent {\bf Case 1(a):} $e_3>0$ and $c_2 = g_2 = 0$.

Since $e_3>0$, Lemma~\ref{matrix}(4) implies that $d_1, e_1, e_2, e_4$ are all zero. So the matrices $M_C$ and $M_D$ are:
\[ 
M_C = \begin{bmatrix} 
a &  &  &  & \\ 
f_1 & 0 &  &  & \\ 
g_1 & h_2 & i_3 &  & \\ 
h_1 & i_2 & i_5 & h_4 & \\ 
i_1 & i_4 & h_3 & g_3 & f_2
\end{bmatrix}   ~~~~M_D = \begin{bmatrix} 
a &  &  &  & \\ 
b_1 & 0 &  &  & \\ 
c_1 & d_2 & e_3 &  & \\ 
0 & 0 & e_5 & d_4 & \\ 
0 & 0 & d_3 & c_3 & b_2
\end{bmatrix}
\]

Both here and throughout, we leave the entries above the main diagonal blank, as they are always zero (Lemma~\ref{matrix}).

Since row 2 of $M_D$ sums to $f$, we have $b_1 = f_1 + f_2$. On the other hand, row 2 of $M_C$ sums to $b$, so $f_1 = b_1+b_2$. Therefore, $0 = b_2 + f_2$, from which it follows that $b_2 = f_2 = 0$ and $b_1 = f_1.$ Column 2 of $M_D$ sums to $f$, so $f_1 = d_2$. 

Because $f_2=b_2=0$ and because these are the only entries of the matrices where the associated corresponding pairs have difference 1, we know that $\deltaC{\ell} \neq 1$ and $\deltaD{\ell} \neq 1$. Because $\deltaC{1} + \deltaC{\ell} = 10$ and $\deltaD{1} + \deltaD{\ell} = 10$ (Lemma~\ref{pairs}), it follows that $\deltaC{1} \neq 9$ and $\deltaD{1} \neq 9$. Therefore $a = 0$.

Let $w=b_1=f_1=d_2$. The matrices $M_C$ and $M_D$ are now as follows:
\[
M_C = \begin{bmatrix} 
0 &  &  &  & \\ 
w & 0 &  &  & \\ 
g_1 & h_2 & i_3 &  & \\ 
h_1 & i_2 & i_5 & h_4 & \\ 
i_1 & i_4 & h_3 & g_3 & 0
\end{bmatrix}   
~~~~M_D = \begin{bmatrix} 
0 &  &  &  & \\ 
w & 0 &  &  & \\ 
c_1 & w & e_3 &  & \\ 
0 & 0 & e_5 & d_4 & \\ 
0 & 0 & d_3 & c_3 & 0
\end{bmatrix}
\]
Because column 5 of $M_D$ sums to $i-z$, we have $i=z$. Similarly, $e=z$. From this, we conclude $e_5 = z-e_3$.  Also, because column 1 of $M_D$ sums to $z$, $c_1 = z-w$.
\[ 
M_C = \begin{bmatrix} 
0 &  &  &  & \\ 
w & 0 &  &  & \\ 
g_1 & h_2 & i_3 &  & \\ 
h_1 & i_2 & i_5 & h_4 & \\ 
i_1 & i_4 & h_3 & g_3 & 0
\end{bmatrix}   
~~~~M_D = \begin{bmatrix} 
0 &  &  &  & \\ 
w & 0 &  &  & \\ 
\mathsmaller{z-w} & w & e_3 &  & \\ 
0 & 0 & \mathsmaller{z-e_3} & d_4 & \\ 
0 & 0 & d_3 & c_3 & 0
\end{bmatrix}
\]
Because row 3 and column 3 of $M_D$ have the same entry sum, it follows that $d_3 = e_3$. Since $d_3 = e_3>0$, we conclude that $d_4=0$ by Lemma~\ref{matrix}(4). Moreover, since row 5 of $M_D$ has a sum of $z$, we conclude that $c_3 = z-e_3$. Setting $e_3 = u$, we have: 
\[ 
M_C = \begin{bmatrix} 
0 &  &  &  & \\ 
w & 0 &  &  & \\ 
g_1 & h_2 & i_3 &  & \\ 
h_1 & i_2 & i_5 & h_4 & \\ 
i_1 & i_4 & h_3 & g_3 & 0
\end{bmatrix}   
~~~~M_D = \begin{bmatrix} 
0 &  &  &  & \\ 
w & 0 &  &  & \\ 
\mathsmaller{z-w} & w & u &  & \\ 
0 & 0 & \mathsmaller{z-u} & 0 & \\ 
0 & 0 & u & \mathsmaller{z-u} & 0
\end{bmatrix}
\]
Next, we claim that $h_1=0$ and $h_3 = i_3$. For, if $i_3=0$, Equation~\ref{1} simplifies to show that $h_1 = h_3 = 0.$ On the other hand, if $i_3>0$, then Lemma~\ref{matrix}(4) implies $h_1=0$, and Equation~\ref{1} simplifies to show that $h_3=i_3$. 

Column 3 of $M_C$ also sums to $c$, which implies that $i_5 = 2z-w-u-2i_3.$ Denoting $i_3$ by $v$, the matrices are:
\[ 
M_C = \begin{bmatrix} 
0 &  &  &  & \\ 
w & 0 &  &  & \\ 
g_1 & h_2 & v &  & \\ 
0 & i_2 & \mathsmaller{2z-w-u-2v} & h_4 & \\ 
i_1 & i_4 & v & g_3 & 0
\end{bmatrix}   
~~~~M_D = \begin{bmatrix} 
0 &  &  &  & \\ 
w & 0 &  &  & \\ 
\mathsmaller{z-w} & w & u &  & \\ 
0 & 0 & \mathsmaller{z-u} & 0 & \\ 
0 & 0 & u & \mathsmaller{z-u} & 0
\end{bmatrix}
\]

Next we claim that $i_1=i_2 = i_4 = 0$. We will prove this one entry at a time. Note that if any of these three entries are non-zero, then $v=0$. Suppose $i_2>0$. By Lemma~\ref{matrix}(4), $i_1 = 0$, which in turn forces $g_1 = z-w$ and $h_2 = z-u$. Equation~\ref{2} simplifies to $i_2 = 2(u+w-z)$, while the fact that column 2 of $M_C$ sums to $w$ implies that $i_4 = z-u-w$. Observe that $i_2 = -2i_4$, which contradicts the assumptions that $i_2>0$ and $i_4\geq0$. Therefore, $i_2 = 0$, and we also obtain from Equation~\ref{2} that $g_1 = 2u+2v+w-z$. Next, suppose that $i_1>0$. Then since column 1 of $M_C$ sums to $z$, we conclude that $i_1 = 2(z-u-w).$ On the other hand, adding $i_k$ for all $k$ should result in $z$, which forces $i_4 = 3(w+u-z)$. Observe that $i_1 = -\frac{2}{3}i_4$. This contradicts the assumption that $i_1>0$ and  $i_4\geq0$.  Therefore, it must be that $i_1=0$. Finally, suppose that $i_4>0$. Then by Lemma~\ref{matrix}(4), we have $i_5 = 2z-w-u=0$ and $h_4 = 0$. Thus, every entry in Row 4 of $M_C$ is 0, but the row should sum to $d = w+u>0$, a contradiction. Therefore, $i_4=0$, and we have:
\[
M_C = \begin{bmatrix}
0 &  &  &  & \\ 
w & 0 &  &  & \\ 
\mathsmaller{2u+2v+w-z} & h_2 & v &  & \\ 
0 & 0 & \mathsmaller{2z-w-u-2v} & h_4 & \\ 
0 & 0 & v & g_3 & 0
\end{bmatrix}  
\; 
M_D = \begin{bmatrix} 
0 &  &  &  & \\ 
w & 0 &  &  & \\ 
\mathsmaller{z-w} & w & u &  & \\ 
0 & 0 & \mathsmaller{z-u} & 0 & \\ 
0 & 0 & u & \mathsmaller{z-u} & 0
\end{bmatrix}
\]
Only a few steps remain. Row 5 of $M_C$ sums to $z$, therefore $g_3 = z-v$. Column 4 and Row 4 of $M_C$ have the same sum, so $w = z-v-u$. Column 4 of $M_C$ sums to $d=w+u$, thus $h_4 = w+u+v-z=0$. Column 2 of $M_C$ sums to $b=w$, so $h_2=w$. Replacing $w$ with $z-v-u$, we have:
\setlength\arraycolsep{4pt}
\[
M_C = \begin{bmatrix}
0 &  &  &  & \\ 
\mathsmaller{z-v-u} & 0 &  &  & \\ 
\mathsmaller{u+v} & \mathsmaller{z-v-u} & v &  & \\ 
0 & 0 & \mathsmaller{z-v} & 0 & \\ 
0 & 0 & v & \mathsmaller{z-v} & 0
\end{bmatrix}   
\;
M_D = \begin{bmatrix}
0 &  &  &  & \\ 
\mathsmaller{z-v-u} & 0 &  &  & \\ 
\mathsmaller{u+v} & \mathsmaller{z-v-u} & u &  & \\ 
0 & 0 & \mathsmaller{z-u} & 0 & \\ 
0 & 0 & u & \mathsmaller{z-u} & 0
\end{bmatrix}
\]
Therefore,
$$\overline{C} = 9_{z}8_{z-v-u}7_{z+v}6_{z-v}5_z4_z3_{z-v}2_{z+v}1_{z-v-u}$$
$$\overline{D} = 9_{z}8_{z-v-u}7_{z+u}6_{z-u}5_z4_z3_{z-u}2_{z+u}1_{z-v-u} $$
where $u>0$, $v\geq0$. We must also require $u\neq v$, for otherwise $C=D$, and we have a Kaprekar constant rather than a 2-cycle. Moreover, notice that if $z = u+v,$ then it would follow that $C$ and $D$ do not contain the digit 1, in which case the smallest digit in both $C$ and $D$ is 2. By Theorem~\ref{smallest}, it follows that $C$ and $D$ contain the same number of digits equal to 2. Hence $u=v$, and once again we have a Kaprekar constant rather than a 2-cycle. Therefore $z>u+v$.\\
\newpage
{\bf Case 1(b):} $e_3>0$, $c_2>0$, and $g_2 = 0$.

Notice that if $i_3>0,$ then the matrices fit into Case 1(c), except with the roles of $C$ and $D$ reversed. Therefore, we may assume that $i_3 = 0$. Because $e_3$ and $c_2$ are both positive, Lemma~\ref{matrix}(4) implies that $c_1$, $d_1$, $e_1$, $e_2$, and $e_4$ are all 0. So, the matrices are as follows:
\[
M_C = \begin{bmatrix}
a & \x & \x & \x & \x\\ 
f_1 & 0 & \x & \x & \x\\ 
g_1 & h_2 & 0 & \x & \x\\ 
h_1 & i_2 & i_5 & h_4 & \x\\ 
i_1 & i_4 & h_3 & g_3 & f_2
\end{bmatrix}  
\;\; 
M_D = \begin{bmatrix}
a & \x & \x & \x & \x\\ 
b_1 & c_2 & \x & \x & \x\\ 
0 & d_2 & e_3 & \x & \x\\ 
0 & 0 & e_5 & d_4 & \x\\ 
0 & 0 & d_3 & c_3 & b_2
\end{bmatrix}
\] 
Column 1 of $M_D$ sums to $a+z$, so then $b_1 = z$. The sum of the entries in Column $i$  of $M_D$ and that of Row $i$ of $M_D$ are the same sum for $i=2,3,$ and 4, respectively. For each respective value of $i$, we conclude the following equations. For $i=2$, $d_2 = z$, and for $i=3$,  $e_5+d_3 = z$, and for $i=4$, $c_3 = e_5$.

Starting with Equation~\ref{2}, we have
\begin{align*}
    e_3 + i_3 &= h_1 + h_3 + d_1 + d_3 + f_2 + b_2\\
    e_3+0 &= h_1 + h_3 + 0 + (z-e_5) + (e-z) + b_2\\
    e_3+e_5 &= h_1 + h_3 + z + (e-z) + b_2\\
    e &= h_1 + h_3 + z + (e-z) + b_2\\
    0 &= h_1 + h_3 + b_2
\end{align*}
Therefore, $h_1=h_3=b_2=0$.
\[
M_C = \begin{bmatrix}
a & \x & \x & \x & \x\\ 
f_1 & 0 & \x & \x & \x\\ 
g_1 & h_2 & 0 & \x & \x\\ 
0 & i_2 & i_5 & h_4 & \x\\ 
i_1 & i_4 & 0 & g_3 & f_2
\end{bmatrix}  
\;\; 
M_D = \begin{bmatrix}
a & \x & \x & \x & \x\\ 
z & c_2 & \x & \x & \x\\ 
0 & z & e_3 & \x & \x\\ 
0 & 0 & e_5 & d_4 & \x\\ 
0 & 0 & \mathsmaller{z-e_5} & e_5 & 0
\end{bmatrix}
\] 
Because row 2 of $M_C$ sums to $b = b_1 + b_2 =  z$, we conclude that $f_1 = z$. Now using the fact that Column 1 of $M_C$ sums to $a+z$, it follows that $g_1=i_1=0$. 
\[
M_C = \begin{bmatrix}
a & \x & \x & \x & \x\\ 
z & 0 & \x & \x & \x\\ 
0 & h_2 & 0 & \x & \x\\ 
0 & i_2 & i_5 & h_4 & \x\\ 
0 & i_4 & 0 & g_3 & f_2
\end{bmatrix}  
\;\; 
M_D = \begin{bmatrix}
a & \x & \x & \x & \x\\ 
z & c_2 & \x & \x & \x\\ 
0 & z & e_3 & \x & \x\\ 
0 & 0 & e_5 & d_4 & \x\\ 
0 & 0 & \mathsmaller{z-e_5} & e_5 & 0
\end{bmatrix}
\] 
Row 2 of $M_D$ should sum to $f = f_1+f_2 = z + f_2$. This forces $f_2 = c_2$.
Next, the sum of Row 3 of $M_D$ should be $g = g_1 + g_2 + g_3 = 0+0+g_3$. Therefore $z+e_3 = g_3$. 

The sum of the bottom row of $M_C$ is then $i_4 + z+e_3 + c_2$, but on the other hand, the sum should be $e_1 + e_2 + e_3 + e_4 + e_5 = e_3+e_5$. Therefore $e_5 = i_4 + z + c_2$. We conclude $e_5 - z = i_4 + c_2\geq 0$, but we also have a matrix entry of $z - e_5\geq 0$. Therefore, it follows that $i_4$ and $c_2$ are both 0. This is a contradiction, because we assumed $c_2 >0$.\\

\newpage
{\bf Case 1(c):} $e_3>0$ and $ g_2 > 0$.

Since $e_3$ and $g_2$ are positive, Lemma~\ref{matrix}(4) implies that $d_1$, $e_1$, $e_2$, $e_4$, $g_1$, $h_1$, $i_1$ are all 0, and our matrices are as follows.
\[ 
M_C = 
\begin{bmatrix} 
a & \\ 
f_1 & g_2 & \\ 
0 & h_2 & i_3 & \\ 
0 & i_2 & i_5 & h_4 & \\ 
0 & i_4 & h_3 & g_3 & f_2
\end{bmatrix} 
\;
M_D = 
\begin{bmatrix} 
a & \\ 
b_1 & c_2 & \\ 
c_1 & d_2 & e_3 &\\ 
0 & 0 & e_5 & d_4 & \\ 
0 & 0 & d_3 & c_3 & b_2
\end{bmatrix}
\]
Observing the first column of $M_C$, we conclude $f_1=z$. Since $f=b_1+c_2$, we have
\begin{align*}
    b_1 + c_2 &= z + f_2 \\
    b_1 + c_1 + c_2 &= z + f_2 + c_1 \\
    z + c_2 &= z + f_2 + c_1 \\
    c_2 &= f_2 + c_1 
\end{align*}
From this relationship, we see that if $c_1>0$, then $c_2>0$, contradicting Lemma~\ref{matrix}(4). Thus $c_1=0$ and, observing column 1 of $M_D$, $b_1=z$. The entries of row 2 and column 2 of $M_D$, respectively, have the same sum. Thus $d_2=z$. This leaves the matrices as follows. 
\[
M_C = 
\begin{bmatrix} 
a & \\ 
z & g_2 & \\ 
0 & h_2 & i_3 & \\ 
0 & i_2 & i_5 & h_4 & \\ 
0 & i_4 & h_3 & g_3 & f_2
\end{bmatrix}   
\;
M_D = 
\begin{bmatrix} 
a & \\ 
z & c_2 &\\ 
0 & z & e_3 & \\ 
0 & 0 & e_5 & d_4 & \\ 
0 & 0 & d_3 & c_3 & b_2
\end{bmatrix}
\]
Row 2 of $M_C$ sums to $b$. Therefore we have $g_2 = b_2$. Similarly, considering row 2 of $M_D$, we conclude $c_2 = f_2.$ The fact that row 3 and column 3 of $M_D$ have the same sum implies that $e_5+d_3 = z$. The fact that row 4 and column 4 of $M_D$ have the same sum implies that $e_5=c_3.$ Therefore $d_3 = z-c_3$.

Suppose $i_2>0$ (we will show that this causes a contradiction). This implies $i_3=0$ by Lemma~\ref{matrix}(4), and Equation~\ref{1} yields: 
     $e_3 = h_3+z-c_3+c_2+g_2.$
Beginning with $c_2=f_2=e-z$ (from the (5,5) entry of $M_C$), we use this expression for $e_3$ in the following:
\begin{align*}
    c_2+z &= e \\
    c_2+z&= e_1+e_2+e_3+e_4+e_5 \\
    c_2+z&= 0+0+(h_3+z-c_3+c_2+g_2)+0 +c_3 \\
    0&= h_3+g_2
\end{align*}
But one of our original assumptions was that $g_2>0$, so this is a contradiction. Therefore, our supposition that $i_2>0$ was incorrect. It must be that $i_2=0$. Then Equation~\ref{2} becomes $2h_3+2d_3=0$. We conclude $h_3=d_3=0$. Since $d_3=z-c_3$, $c_3=z$. Our matrices are as follows:
\[ 
M_C = \begin{bmatrix} 
a & \\ 
z & g_2 & \\ 
0 & h_2 & i_3 &  \\ 
0 & 0 & i_5 & h_4 & \\ 
0 & i_4 & 0 & g_3 & c_2
\end{bmatrix}   
\; M_D = 
\begin{bmatrix} 
a &  \\ 
z & c_2 &  \\ 
0 & z & e_3 &  \\ 
0 & 0 & z & d_4 & \\ 
0 & 0 & 0 & z & g_2
\end{bmatrix}
\]
Since column 3 and row 3 of $M_C$ have the same sum, $h_2 = i_5$. 
Because column 4 and row 4 of $M_C$ have the same sum, $h_2=i_5=g_3$. 
Finally, since column 2 and row 2 of $M_C$ have the same sum, $z= h_2 + i_4$

Combining $i_5=h_2$ and $z=h_2+i_4$, we conclude $i=i_3+z$. On the other hand,  row 5 of $M_D$ should also have a sum of $i$. Therefore, $i_3 = g_2$. A parallel argument establishes $c_2=e_3$. 

Since $g_2>0$, we know $i_3>0$. Lemma~\ref{matrix}(4) implies $i_4=0$, so $h_2=z$. The matrices are then: 
\[ 
M_C = 
\begin{bmatrix} 
a &  \\ 
z & g_2 &  \\ 
0 & z & g_2 &  \\ 
0 & 0 & z & h_4 & \\ 
0 & 0 & 0 & z & c_2
\end{bmatrix}   
\;
M_D = 
\begin{bmatrix} 
a &  \\ 
z & c_2 &  \\ 
0 & z & c_2 & \\ 
0 & 0 & z & d_4 & \\ 
0 & 0 & 0 & z & g_2
\end{bmatrix}
\]
Observe that the entries of the matrices coincide on all entries off of their main diagonals. We will show that the main diagonal entries coincide as well. On one hand, if we add all the entries $d_j$ in $M_D$, we find $d=z+d_4$. On the other hand, row 4 of $M_C$ also has a sum of $d$, so $d_4=h_4$. Similarly, if we add all the entries $c_j$ in $M_D$, we have $c=z+c_2$. On the other hand, the entries of row 3 of $M_C$ sum to $c$. Therefore, $g_2=c_2$. This proves that $M_C=M_D$, and thus $C=D$, so we have a Kaprekar constant, rather than a 2-cycle. \\

\noindent {\bf Case 2:} $e_3=0$ and $i_3= 0$.

Setting $e_3=0$ and $i_3= 0$ in Equation~\ref{1}, we conclude $h_1, h_3, d_1, d_3, f_2,$ and $b_2$ are all 0. Next, setting $h_3=0$ and $d_3=0$ in Equation~\ref{2}, we conclude $i_2, e_2, g_1,$ and $c_1$ are all 0. The matrices are then given below. 
\[
M_C = 
\begin{bmatrix}
a & \x & \x & \x & \x\\ 
f_1 & g_2 & \x & \x & \x\\ 
0 & h_2 & 0 & \x & \x\\ 
0 & 0 & i_5 & h_4 & \x\\ 
i_1 & i_4 & 0 & g_3 & 0
\end{bmatrix}  
\;\; 
M_D = 
\begin{bmatrix}
a & \x & \x & \x & \x\\ 
b_1 & c_2 & \x & \x & \x\\ 
0 & d_2 & 0 & \x & \x\\ 
0 & 0 & e_5 & d_4 & \x\\ 
e_1 & e_4 & 0 & c_3 & 0
\end{bmatrix}
\] 
The entries of row 3 and the entries column 3 of $M_C$ each sum to $c$, so $h_2=i_5=c$. Similarly, the entries of row 3 and column 3 of $M_D$ each sum to $g$, so $d_2=e_5=g$.  

Next, row 4 of $M_C$ sums up to $d$ so then $h_4=d-c$, and row 4 of $M_D$ sums up to $h$, so then $d_4=h-g$. Then, since $f_2=0$, and $f=f_1+f_2$, we know $f=f_1$, and by the same logic, $b_1=b$. 

Because row 2 of $M_C$ sums up to $b$, we know $g_2=b-f\geq 0$. As row 2 of $M_D$ sums to $f$, $c_2=f-b \geq 0$. Therefore $f=b$ and $g_2=c_2=0$. This forces $g_3 = g$ and $c_3 = c$. The updated matrices are as follows:
\[
M_C = 
\begin{bmatrix}
a & \x & \x & \x & \x\\ 
f & 0 & \x & \x & \x\\ 
0 & c & 0 & \x & \x\\ 
0 & 0 & c & \mathsmaller{d-c} & \x\\ 
i_1 & i_4 & 0 & g & 0
\end{bmatrix}  
\;\; 
M_D = 
\begin{bmatrix}
a & \x & \x & \x & \x\\ 
f & 0 & \x & \x & \x\\ 
0 & g & 0 & \x & \x\\ 
0 & 0 & g & \mathsmaller{h-g} & \x\\ 
e_1 & e_4 & 0 & c & 0
\end{bmatrix}
\] 
Column 2 of $M_C$ sums to $b$, so $i_4=b-c = f-c$. Column 2 of $M_D$ sums to $f$ so $e_4=f-g$.  Next, column 1 of $M_C$ sums to $a+z$, 
so $i_1=z-f$. Similarly, $e_1=z-f$. 
The entries in the fourth column of $M_D$ sum to $h$, so $g=c.$ The updated matrices are as follows:
\[
M_C = \begin{bmatrix}
a & \x & \x & \x & \x\\ 
f & 0 & \x & \x & \x\\ 
0 & c & 0 & \x & \x\\ 
0 & 0 & c & \mathsmaller{d-c} & \x\\ 
\mathsmaller{z-f} & \mathsmaller{f-c} & 0 & c & 0
\end{bmatrix}  
\;\; 
M_D = \begin{bmatrix}
a & \x & \x & \x & \x\\ 
f & 0 & \x & \x & \x\\ 
0 & c & 0 & \x & \x\\ 
0 & 0 & c & \mathsmaller{h-c} & \x\\ 
\mathsmaller{z-f} & \mathsmaller{f-c} & 0 & c & 0
\end{bmatrix}
\] 
Finally, $d = d_1+d_2+d_3+d_4 = 0 + c + 0 + h-c= h$. We conclude $M_C=M_D$, which implies $C=D$, and we have a Kaprekar constant rather than a 2-cycle. 
\end{proof}

\section{Conclusion}
As mentioned in the introduction, Dolan~\cite{dolan0} has since found the full classification of Kaprekar 2-cycles. He proved that all Kaprekar 2-cycles are exactly those listed in Theorem~\ref{MainResult}. We mention some other interesting (but as of yet unfinished) storylines here. For a broad overview of related results and some open questions, we recommend~\cite{devlinzeng} and~\cite{Yamagami}. 

In our discussion, we focused on the Kaprekar process with numbers written in base 10. The Kaprekar process has also been studied with respect to numbers written in any base $b$. For example, Eldridge-Sagong~\cite{eldridge} proved that for each {\em odd} base $b$, there exists a Kaprekar 2-cycle, namely the following pair of 3-digit numbers: 
$$C=\left(\frac{b-3}{2}\right)(b-1)\left(\frac{b+1}{2}\right)~~ \text{ and } ~~D=\left(\frac{b-1}{2}\right)(b-1)\left(\frac{b+1}{2}\right).$$
Eldridge-Sagong proved that every 3-digit number written in an odd base $b>3$ will reach the above 2-cycle in at most $\frac{b+1}{2}$ steps of the Kaprekar process (1 step for $b=3$). 

Efforts to classify Kaprekar $n$-cycles for $n>2$ both in  base 10 and other bases include~\cite{dolan4, dolan3, eldridge, trigg}. Computer testing of numbers up to 140 digits long in base 10 reveals Kaprekar cycles of length 1, 2, 3, 4, 5, 7, 8 and 14~\cite{oeis}. This is an indication that a complete classification of all Kaprekar cycles in base 10 is within reach. 

\bibliographystyle{plain}
\bibliography{biblio}

\end{document}